\documentclass[a4paper,11pt,twoside]{article}

\usepackage[left=1.1in,right=1.1in,top=1.5in,bottom=1.5in]{geometry}
\usepackage{lmodern}
\usepackage[utf8]{inputenc}
\usepackage[T1]{fontenc}
\usepackage[english]{babel}
\usepackage{mathtools,amsmath,amssymb,amsfonts,amsthm,mathrsfs}
\usepackage{bbm}
\usepackage[title,titletoc]{appendix}

\usepackage{pgfplots}
\usepackage{graphicx}
\graphicspath{{Figures/}}
\usepackage{caption}
\usepackage{subcaption}
\usepackage{microtype}
\usepackage{bigints}
\usetikzlibrary{decorations.markings}
\usetikzlibrary{patterns}

\pgfplotsset{compat=1.12}

\usepackage{booktabs}
\usepackage{multirow}
\usepackage{rotating,tabularx}
\usepackage{adjustbox}
\usepackage{enumitem}

\usepackage{cases}

\definecolor{bostonuniversityred}{rgb}{0.7, 0.0, 0.0}
\definecolor{forestgreen}{rgb}{0.13, 0.45, 0.13}
\definecolor{blue-violet}{rgb}{0.5, 0.1, 0.8}
\definecolor{blue-new}{rgb}{0.2, 0.2, 0.8}
\definecolor{amber}{rgb}{1.0, 0.49, 0.0}
\definecolor{applegreen}{rgb}{0.55, 0.71, 0.0}

\newtheorem{theorem}{Theorem}[section]
\newtheorem{definition}{Definition}[section]
\newtheorem{lemma}[theorem]{Lemma}
\newtheorem{proposition}[theorem]{Proposition}
\newtheorem{corollary}[theorem]{Corollary}

\theoremstyle{definition}

\theoremstyle{definition}
\newtheorem{remark}[theorem]{Remark}
\numberwithin{equation}{section}

\usepackage[hidelinks]{hyperref}
\usepackage{fancyhdr}
\makeatletter
\def\thanks#1{\protected@xdef\@thanks{\@thanks
        \protect\footnotetext{#1}}}
\makeatother
\renewenvironment{abstract}[1]
{\list{}{\setlength{\leftmargin}{3em}
 \setlength{\rightmargin}{\leftmargin}}\item[]
\textbf{\abstractname.} #1\relax}
{\endlist}
\providecommand{\thanksbody}
{\hspace{-2em}
E. Mu\~noz-Hern\'andez is supported by the Research Grant PID2021-123343NB-I00 of the Ministry of Science, Technology, and Universities of Spain. 

E. Sovrano acknowledges the PRIN 2022 project \emph{Modeling, Control and
Games through Partial Differential Equations} (D53D23005620006),
funded by the European Union - Next Generation EU.

V. Taddei is supported by the Grant MIUR-PRIN 2020F3NCPX ``Mathematics for industry 4.0 (Math4I4)''.

E. Sovrano and V. Taddei are members of the \emph{Grup\-po Na\-zio\-na\-le per l'Anali\-si Ma\-te\-ma\-ti\-ca, la Pro\-ba\-bi\-li\-t\`{a} e le lo\-ro Appli\-ca\-zio\-ni} (GNAMPA) of the \emph{Isti\-tu\-to Na\-zio\-na\-le di Al\-ta Ma\-te\-ma\-ti\-ca} (INdAM) and acknowledge financial support from this institution. 

}

\title{\textbf{Coupled reaction-diffusion equations with degenerate diffusivity: wavefront analysis}}

\author{Eduardo Mu\~noz-Hern\'andez, Elisa Sovrano, Valentina Taddei
\thanks{\thanksbody}
}
\begin{document}

\newcommand{\Addresses}{
\bigskip
\small

\noindent Eduardo Mu\~noz-Hern\'andez\\
Universidad Complutense de Madrid,\\
Instituto de Matem\'atica Interdisciplinar (IMI),\\
Departamento de An\'alisis Matem\'atico y Matem\'atica Aplicada,\\
Plaza de las Ciencias 3, 28040 Madrid, Spain\\
ORCID id: 0000-0003-1184-6231\\
email: \href{mailto:eduardmu@ucm.es}{\texttt{eduardmu@ucm.es}}

\medskip

\noindent Elisa Sovrano\\
Dipartimento di Scienze e Metodi dell'Ingegneria,\\
Universit\`a degli Studi di Modena e Reggio Emilia,\\
Via G. Amendola 2, 42122 Reggio Emilia, Italy\\
email: \href{mailto:elisasovrano@unimore.it}{\texttt{elisasovrano@unimore.it}}

\medskip

\noindent Valentina Taddei\\
Dipartimento di Scienze e Metodi dell'Ingegneria,\\
Universit\`a degli Studi di Modena e Reggio Emilia,\\
Via G. Amendola 2, 42122 Reggio Emilia, Italy\\
email: \href{mailto:valentina.taddei@unimore.it}{\texttt{valentina.taddei@unimore.it}}

}

\maketitle
\thispagestyle{empty}
%%================================

%%================================
{\small{
\begin{abstract}
\noindent We investigate traveling wave solutions for a nonlinear system of two coupled reaction-diffusion equations characterized by double degenerate diffusivity:
\[n_t= -f(n,b), \quad b_t=[g(n)h(b)b_x]_x+f(n,b).\]
These systems mainly appear in modeling spatio-temporal patterns during bacterial growth.
Central to our study is the diffusion term $g(n)h(b)$, which degenerates at $n=0$ and $b=0$; and the reaction term $f(n,b)$, which is positive, except for $n=0$ or $b=0$. Specifically, the existence of traveling wave solutions composed by a couple of strictly monotone functions for every wave speed in a closed half-line is proved, and some threshold speed estimates are given. Moreover, the regularity of the traveling wave solutions is discussed in connection with the wave speed. 

\textbf{Mathematics Subject Classifications:} 35C07, 35K55, 35K57.

\textbf{Keywords:} degenerate diffusion, coupled reaction-diffusion equations, traveling wave solution, wave speed, sharp profile.
\end{abstract}}
}
%%================================
\bigskip
%%================================
\section{Introduction}\label{section-1}
%%================================
In this paper, we investigate the system of coupled reaction-diffusion equations:
\begin{subnumcases}{\label{eq-sys-pde}}
n_t= -f(n,b), \label{sys-ode1}\\
b_t=[g(n)h(b)b_x]_x+f(n,b). \label{sys-ode2}
\end{subnumcases}
Here, $n=n(x,t)$ and $b=b(x,t)$ denote two unknown functions dependent on time $t\geq0$ and space $x\in\mathbb{R}$. The terms $g(n)h(b)$ and $f(n,b)$ are representative of the diffusivity and the reaction, respectively. We focus on scenarios where the diffusivity doubly degenerates at zero, characterized by $g$ and $h$ vanishing when $n=0$ and $b=0$, respectively. 

Models related to~\eqref{eq-sys-pde} have played a crucial role in exploring the growth dynamics of the bacterial species named \textit{Bacillus subtilis} on agar plates, as reported in~\cite{BCL-00, KMUS-97, MSM-00, SMGA-01} where $n$ and $b$ stand for the nutrient and bacterial concentrations, respectively. All these works provided insights into studying spatiotemporal pattern formations, particularly when diffusivity is a proportional function of the product between the nutrient and bacterial concentrations. For instance, numerical analyses in~\cite{SMGA-01}, which specifically considered the case where $g(n)h(b)=nb$ and $f(n,b)=nb$, indicates the presence of traveling waves. These are solutions that support profiles with a constant speed of propagation. The research in~\cite{SMGA-01} highlighted the existence of a critical propagation speed where the traveling wave is non-differentiable at a certain point, referred to as a sharp traveling wave.

While significant focus has been on certain diffusivity and reaction choices, a comprehensive exploration of the existence and qualitative analysis of traveling waves for equation \eqref{eq-sys-pde} seems to be limited or largely unexplored. It is worth noting that other reactions such as $f(n,b)=\frac{nb}{1+kn}$ (with $k\in\mathbb{R}$) and diffusivities involving strictly increasing functions $g$ and $h$ vanishing only at $0$ are biologically relevant, as outlined in~\cite{BaDe-07, KMUS-97,Ha-04, Mu-02,Maini-new}. Also, in those cases, to the best of our knowledge, comprehensive mathematical analyses addressing traveling waves for systems as~\eqref{eq-sys-pde} are still lacking.

In response, we deal with system~\eqref{eq-sys-pde} by considering a wide range of diffusion and reaction terms, including all the examples above. The assumptions on those terms are detailed in Section~\ref{subsec-intro1}. We establish the existence of a half-line of admissible wave speeds with a unique associated traveling wave. Notably, all the traveling waves we found are composed of a couple of functions that are strictly monotone. We also discuss their regularity, and besides giving estimates of the threshold speed, we prove that the profiles are smooth functions, except for the one at the threshold speed, which exhibits a ``sharp'' behavior. These results align with those presented in~\cite{MS-24,SMGA-01} and mirror the behavior in the case of scalar equations with degenerate diffusivity and reaction terms of Fisher-KPP type~\cite{MaMa-03,Maini-JDE95,Maini-new}.

Besides the degenerate models we are dealing with in this paper, the existing literature showcases various applications that exploit reaction-diffusion systems similar to~\eqref{eq-sys-pde}. Some models use constant diffusivity, as in~\cite{Lo-97,LoLo-96}. Others assume diffusion terms that are not doubly degenerate, with assumptions such as $g(n)>0$ and $h(b)=0$ only at $b=0$~\cite{Coall-21,GaMa-22} or at $b=1$~\cite{Mitra-23}. These formulations have been pivotal in modeling combustion, chemical reactions, tumor growth dynamics, and cellulolytic biofilm growth. In all these models the diffusivity and reaction assumptions do not affect the existence of an admissible line of wave speeds. However, the different assumptions on the diffusivity and reaction deeply influence the qualitative properties of the associated traveling waves. As a consequence, these models might support traveling waves that are made by a couple of monotonic smooth functions \cite{GaMa-22, Lo-97,LoLo-96}. Alternatively, they may produce traveling waves that become sharp at the threshold speed~\cite{Coall-21}, similar to our findings, or may produce traveling waves made by a couple of functions that do not maintain monotonic properties~\cite{Mitra-23}.

%------------------------------------------------------------------------
\subsection{Traveling waves: assumptions and statement of main results}\label{subsec-intro1}
%------------------------------------------------------------------------

Our focus in this paper is on traveling waves of~\eqref{eq-sys-pde}, which are solutions that maintain a constant speed $c\in\mathbb{R}$. By introducing the traveling wave coordinate $\xi=x-ct$, we can express $\eta(\xi)=n(x,t)$ and $\beta(\xi)=b(x,t)$. This allows us to analyze the following system:
\begin{subnumcases}{\label{eq-sys}}
	c\eta' - f(\eta,\beta)=0, \label{eq-ode-1} \\
	\left(g(\eta)h(\beta)\beta'\right)'+c\beta'+f(\eta,\beta)=0,\label{eq-ode-2}
\end{subnumcases}
where $'=\frac{\mathrm{d}}{\mathrm{d}\xi}$. Taking into consideration the applications, we henceforth assume that:
\begin{enumerate}[labelwidth=20pt, align=left]
\item[$(A_1)$] $f\colon[0,1]^2\to[0,+\infty)$ is a function of class $C^1$ and such that $f(s,r)=0$ if and only if either $s=0$ or $r=0$, and there exists $0 < L_1\leq L_2$ such that $L_1sr\leq f(s,r)\leq L_2 sr$ for all $(s,r)\in(0,1)^2$;  
\item[$(A_2)$]  $g\colon[0,1]\to[0,+\infty)$ is a function of class $C^1$ such that $g(s)=0$ if and only if $ s=0$ and there exists $ M_g >0 $ such that $g(s)\geq M_g g(s_1)$ for all $s,s_1\in[0,1]$ with $s\geq s_1\geq0$ and $\dot g(0)>0$; 
\item[$(A_3)$] $h\colon[0,1]\to[0,+\infty)$ is a functions of class $C^1$ such that $ h(r)=0$ if and only if $r=0$ and $\dot h(0)>0$.
\end{enumerate}
It should be noted that the regularity assumptions outlined above can be relaxed by requiring them only in a neighborhood of $(s,r)$ with $s,r\in\{0,1\}$.  For the sake of exposition, we assume them in all domains while emphasizing the minimum assumptions required in each section.

\medskip

Since there is a degenerate diffusion term in~\eqref{eq-sys} we can expect traveling waves to have sharp behaviors. For this reason, we clarify below the notion of the solution we consider here.

\begin{definition}[Traveling wave]\label{def-wave}
Let $c\in\mathbb{R}$. A traveling wave of system~\eqref{eq-sys-pde} is a couple of functions $\left(n,b \right)$ such that $n(x,t)=\eta(x-ct)=\eta(\xi)$ and $b(x,t)=\beta(x-ct)=\beta(\xi)$ where $\eta,\beta\colon \mathbb{R}\to[0,1]$ satisfy
\begin{enumerate}[nosep,wide=0pt, labelwidth=15pt, align=left]
\item[$(i)$] $\eta$ is of class $C^1$ and solves~\eqref{eq-ode-1};
\item[$(ii)$] $\beta$ is continuous and differentiable a.e. with $h(\beta)\beta'\in L^1_{loc}(\mathbb{R})$ and solves weakly~\eqref{eq-ode-2}, namely, for every $\psi\in C^{\infty}_0(\mathbb{R})$, it satisfies
\begin{equation}\label{eq-weak}
\int_{\mathbb{R}} \Big(\left(g(\eta(\xi))h(\beta(\xi))\beta'(\xi)+c\beta(\xi)\right)\psi'(\xi) -f(\eta(\xi),\beta(\xi))\psi(\xi) \Big)\,\mathrm{d}\xi=0.
\end{equation} 
\end{enumerate}
The couple $(\eta,\beta)$ is the {wave profile}, and $c$ is the {wave speed}. 
\end{definition}
\noindent\textit{Warning.} The term ``traveling wave'' will be used also for the profile $(\eta,\beta)$ for the sake of brevity.

\begin{remark} According to the definition, a traveling wave takes values in $ [0,1]. $ Therefore, without loss of generality, we denote by $ f $ also its continuous extension in $ \mathbb R^2 $ and by $ g $ and $ h $ their continuous extensions in $ \mathbb R. $ 
\end{remark}

We are interested in profiles $(\eta,\beta)$ connecting the steady states of system~\eqref{eq-sys} which are $(0,\bar{\beta})$ with $\bar{\beta}\in[0,1]$ and $(\bar{\eta},0)$ with $\bar{\eta}\in[0,1]$. For simplicity and in alignment with application-specific requirements as discussed in~\cite{SMGA-01}, we will consider $\bar{\eta}=1=\bar{\beta}$, which leads to the boundary conditions:
\begin{subnumcases}{\label{eq-bc}}
	\left(\eta(-\infty),\beta(-\infty)\right)=(0,1), \label{eq-bc1} \\
	\left(\eta(+\infty),\beta(+\infty)\right)=(1,0), \label{eq-bc2}
\end{subnumcases}
where for any function $\rho$ having limit at $\pm\infty$, we abbreviate $\lim_{\xi\to\pm\infty} \rho(\xi)=\rho(\pm\infty).$

\smallskip 
Additionally, we classify these traveling waves into two categories based on whether they reach equilibrium in finite time or not. 

\begin{definition}[Classical and sharp traveling wave]\label{def-2}
A traveling wave $(\eta,\beta)$ is classical if the function $\beta$ is twice continuously differentiable in $\mathbb{R}$.
On the other hand, a traveling wave $(\eta,\beta)$ is sharp at the point $\ell\in\{0,1\}$ if there exists a real number $\xi_\ell$ such that $\beta(\xi_\ell)=\ell$, and the function $\beta$ is classical on the set $\mathbb{R}\setminus\{\xi_\ell\}$, but not differentiable at the point $\xi_\ell$.
\end{definition}

With these notions in mind, the following theorem collects the core findings of this paper. 

\begin{theorem}
Assume $(A_1)$--$(A_3)$. Then there exists $c_0\in(0,+\infty)$ such that
\begin{itemize}[nosep, leftmargin=*]
\item if $c<c_0$, then~\eqref{eq-sys-pde} has no traveling waves;
\item if $c\geq c_0$, then~\eqref{eq-sys-pde} has a unique (up to shifts) traveling wave whose profile $(\eta,\beta)$ is component-wise monotone, satisfies~\eqref{eq-bc}, and has speed $c$. 
\end{itemize}
Furthermore, the profile $(\eta,\beta)$ is sharp at $0$ if and only if $c=c_0$; otherwise, it is classical. By denoting
\[\begin{split}
&c_{\sharp}:=\max \left\{ \sqrt{L_1M_g \int_0^1 g(1-r)h(r)r\,\mathrm{d}r}, \sqrt{2L_1M_g \int_0^1 (1-r)g(1-r)h(r)r\,\mathrm{d}r} \right\},\\
&c_*:=2\sqrt{L_2\max_{s\in[0,1]}g(s)\sup_{r\in(0,1]}\frac{h(r)}{r}},
\end{split}\]
the following estimates hold:
\[c_{\sharp}\leq c_0\leq c_*.
\]
\end{theorem}

Specifically, it establishes the existence of a half-line of admissible wave speeds with a unique associated traveling wave. Notably, this traveling wave exhibits ``sharp'' behavior exclusively at the threshold speed $c_0$. We will refer to Remark~\ref{rem41} for comparing the estimates $c_{\sharp}$ and  $c_*$ of the threshold speed.

%------------------------------------------------------------------------
\subsection{Organization of the paper}
%------------------------------------------------------------------------

In Section~\ref{section-2}, we discuss the preliminary properties of system~\eqref{eq-sys}--\eqref{eq-bc}. To streamline the subsequent analysis, an equivalent formulation of this system is introduced via equation~\eqref{eq-aux1}.

In Section~\ref{section-3}, we establish the existence and uniqueness, for sufficiently large speeds, of a classical traveling wave whose profile is component-wise monotone (see Theorem~\ref{th-positive}). Firstly, we use the shooting method and fixed-point theory in locally convex topological vector space to prove the existence of such a traveling wave on the negative half-line for every positive wave speed (see Section~\ref{sub-ex-n}). The uniqueness of this traveling wave on the negative half-line is proved using the central manifold theorem and the proof that there is at most one central manifold (since, in general, it is not necessarily unique). Using a comparison technique, we ultimately extend this traveling wave on the positive half-line, yielding a global traveling wave (see Section~\ref{sub-ex-p}).

Finally, in Section~\ref{section-4}, we prove the existence of a threshold speed and explore the qualitative properties of these traveling waves. These properties are analyzed using a first-order reduction method.

%%================================
\section{Preliminary results}\label{section-2}
%%================================
In this section, we prove some properties of traveling waves of~\eqref{eq-sys}--\eqref{eq-bc} that will be crucial in the forthcoming sections. Although we implicitly assume conditions $(A_1)$--$(A_3)$, it is worth noting that, regarding the regularity of function $f$, its Lipschitz continuity would be sufficient.

\begin{proposition}\label{prop1}
If $(\eta,\beta)$ is a traveling wave of~\eqref{eq-sys} and satisfies either~\eqref{eq-bc1} or~\eqref{eq-bc2}, then $ \beta $ satisfies a.e. 
\begin{equation}\label{eq-aux1}
g(\eta)h(\beta)\beta'+c\beta+c\eta-c=0.
\end{equation}
Moreover, if $ J= \{ \xi \in \mathbb R \colon \beta(\xi) = 0\},$ then $\beta \in C^2(\mathbb R \setminus J)$ and satisfies~\eqref{eq-aux1} in $ \mathbb R \setminus J. $
\end{proposition}

\begin{proof}
Let $(a_1,a_2)$ be a bounded interval.  Since $\beta$ is differentiable a.e and $ h(\beta)\beta' \in L^1_{loc} (\mathbb R), $ then $g(\eta)h(\beta)\beta'+c\beta +c\eta \in L^1(a_1,a_2).$ %Denote the null function as $ v $ and 
Assume that $ (\eta,\beta) $ is a traveling wave solution satisfying either~\eqref{eq-bc1} or~\eqref{eq-bc2}. Substituting $f(\eta,\beta) = c \eta'$ in \eqref{eq-weak} and denoted by $ w $ the null function, we obtain, for every $\psi\in C^{\infty}_0(a_1,a_2)$,
\[\begin{split}
\int_{a_1}^{a_2} w(\xi) \psi(\xi) d\xi = 
0= & \int_{a_1}^{a_2} \Big(\left(g(\eta(\xi))h(\beta(\xi))\beta'(\xi)+c\beta(\xi)\right) \psi'(\xi) -c\eta'(\xi)\psi(\xi) \Big)\,\mathrm{d}\xi  \\
=&  \int_{a_1}^{a_2}\left(g(\eta(\xi))h(\beta(\xi))\beta'(\xi)+c\beta(\xi) +c\eta(\xi)\right)\psi'(\xi)\,\mathrm{d}\xi - [c\eta(\xi)\psi(\xi)]_{a_1}^{a_2}  \\
=&  \int_{a_1}^{a_2} \left(g(\eta(\xi))h(\beta(\xi))\beta'(\xi)+c\beta(\xi) +c\eta(\xi)\right)\psi'(\xi)\,\mathrm{d}\xi. 
\end{split}\]
Therefore $g(\eta)h(\beta)\beta'+c\beta +c\eta \in W^{1,1}(a_1,a_2)$ and $\left(g(\eta)h(\beta)\beta'+c\beta +c\eta\right)' =0 $ a.e. It follows that there exists a continuous function $\gamma$ such that $\gamma=g(\eta)h(\beta)\beta'+c\beta +c\eta$ a.e. in $(a_1,a_2)$ and $\gamma(a_2)-\gamma(a_1)=\int_{a_1}^{a_2} w(\xi)\, \mathrm{d}\xi = 0$. Thus, for a.a. $a_1,a_2 \in\mathbb{R},$ $\left(g(\eta)h(\beta)\beta'+c\beta +c\eta\right)(a_2)=\left(g(\eta)h(\beta)\beta'+c\beta +c\eta\right)(a_1).$ From the boundary conditions, we get~\eqref{eq-aux1}. 

To prove that $\beta$ is of class $C^2$ in $\mathbb R \setminus J$, we consider two cases: $\eta(\xi_0)=0$, for some $\xi_0\in\mathbb{R}$, or $\eta(\xi)>0$, for all $\xi\in\mathbb{R}$. In the first case,
$\eta'(\xi_0) =  {f(\eta(\xi_0),\beta(\xi_0))}/{c} = 0$, making $\eta$ a solution of the initial value problem
\[\begin{cases}
	c\eta' - f(\eta,\beta)=0, \\
	\eta(\xi_0)=0, \quad \eta'(\xi_0)=0,
\end{cases}\]
which, by the local Lipschitz continuity of $f$, has only the trivial solution ($\eta\equiv0$). Thus, from $(A_2)$, the continuity of $\beta$, and~\eqref{eq-aux1}, it follows that $\beta\equiv1$, so that $\beta$ is of class $C^2$ in $\mathbb R$. In the second case, given the continuity of $\beta, \eta, g$, and $h$, for every $\xi_0 \not\in J$, there exists $ \epsilon>0 $ such that $ g(\eta)h(\beta) $ never vanishes in $ (\xi_0 - \epsilon, \xi_0 + \epsilon) $ and, applying the mean value theorem, it follows that
\[ \lim_{\delta \to 0} \frac {\beta(\xi_0+\delta) - \beta(\xi_0)}{\delta} = \lim_{\delta \to 0} \frac{1}{\delta} \int_{\xi_0}^{\xi_0 + \delta} \frac {c(1 - \beta(s) - \eta(s))} {g(\eta(s)) h(\beta(s))}\,  \mathrm{d}s =  \frac {c(1 - \beta(\xi_0) - \eta(\xi_0))} {g(\eta(\xi_0)) h(\beta(\xi_0))}, \]
yielding that $\beta$ is differentiable in $\mathbb{R} \setminus J$, and
\begin{equation} \label{beta'}
\beta'(\xi) = \frac {c(1 - \beta(\xi) - \eta(\xi))} {g(\eta(\xi)) h(\beta(\xi))}, 
\end{equation}
for every $ \xi \notin J. $
Since $g$ and $h$ are $C^1$ functions according to $(A_2)$ and $(A_3)$, respectively, the thesis is proven.
\end{proof}

\begin{lemma}\label{lemma0}
If $(\eta,\beta)$ is a traveling wave of~\eqref{eq-sys}--\eqref{eq-bc}, then $c>0.$
\end{lemma}
\begin{proof}
Integrating~\eqref{eq-ode-1} in $(-\infty,+\infty)$ and using~\eqref{eq-bc}, we obtain 
\begin{equation}\label{c1}
c=c\left(1-0\right)=c\int^{+\infty}_{-\infty}\eta'(\xi)\,\mathrm{d}\xi=\int_{-\infty}^{+\infty}f(\eta(\xi),\beta(\xi))\,\mathrm{d}\xi\geq0.
\end{equation}
Suppose, for the sake of contradiction, that $c=0$. Then, by assumption $(A_1)$ and since $\eta$ and $\beta$ are continuous, we would have $f(\eta(\xi),\beta(\xi))=0$ for every $\xi\in\mathbb{R}$. As a result, either $\eta\equiv0$ or $\beta\equiv0$, which contradicts the boundary conditions~\eqref{eq-bc}.
\end{proof}

\begin{lemma}\label{lemma1}
If $(\eta,\beta)$ is a traveling wave of~\eqref{eq-sys}--\eqref{eq-bc}, then there exists $\tau\in\mathbb{R}\cup\{+\infty\}$ such that
\begin{enumerate}[nosep,wide=0pt,  labelwidth=20pt, align=left]
\item[$(i)$] $\beta'(\xi)<0$ and $0<\beta(\xi)<1$, $\forall\xi<\tau$;
\item[$(ii)$] $\eta'(\xi)>0$ and $0<\eta(\xi)<1$, $\forall\xi<\tau$;
\item[$(iii)$] $c=\int_{-\infty}^{\tau} f(\eta(\xi),\beta(\xi))\, \mathrm{d}\xi$ and $\beta(\xi)=0$, $\eta(\xi)= 1$, $ \forall \,\xi \geq \tau$;
\item[$(iv)$] $\eta(\xi)+\beta(\xi)\geq1,$ $\forall\xi\in\mathbb{R}$.
\end{enumerate}
\end{lemma}

\begin{proof}
First, we prove $(i)$--$(ii)$. Notice that by definition of traveling wave $\eta(\xi),\beta(\xi)\in[0,1]$, for every $\xi\in\mathbb{R}$. Using the same reasoning as in the proof of Proposition~\ref{prop1} and since~\eqref{eq-bc2} holds, we can show that $\eta(\xi) > 0$ for every $\xi\in\mathbb{R}$. Now, since $\beta$ is continuous and satisfies~\eqref{eq-bc1}, we define $\tau\in\mathbb{R}\cup\{+\infty\}$ such that 
\[
\beta(\tau)=0 \quad\text{and}\quad \beta(\xi)>0, \quad \forall \xi\in(-\infty,\tau).
\]
Hence, according to Proposition~\ref{prop1}, $\beta$ is of class $C^2$ in $(-\infty,\tau).$ Now, we claim that $\beta'(\xi)<0$, for every $\xi\in(-\infty,\tau)$. Let us suppose, by contradiction, that there exists $\xi_1 < \tau $ such that $\beta'(\xi_1)=0.$ By considering the function $\Lambda(\xi)=\left(g(\eta)h(\beta)\beta' \right)(\xi)$, we observe that $\Lambda(\xi_1)=0$ and $\Lambda'(\xi_1)<0.$ Thus, it follows that, for some $\epsilon>0$, $\Lambda(\xi)>0$ for every $\xi\in(\xi_1-\epsilon,\xi_1)$ and $\Lambda(\xi)<0$ for every $\xi\in(\xi_1,\xi_1+\epsilon)$. Since $\mathrm{sign}(\Lambda)=\mathrm{sign}(\beta')$ in $(-\infty,\tau)$, we deduce that $\xi_1$, as well as every other stationary point of $\beta$, is a local point of strict maximum for $\beta$ which is a contradiction. Finally, the boundary conditions~\eqref{eq-bc} prove the claim, so $(i)$ holds.

From~$(A_1)$, \eqref{eq-ode-1}, and~\eqref{c1}, we observe that $\eta'(\xi) \geq0$, for every $\xi\in\mathbb{R}$, and $\eta'(\xi) = 0 $ if and only if $ \eta(\xi)\beta(\xi)=0.$ Hence, $ \eta'(\xi)=0 $ if and only if $ \beta(\xi)=0.$ This leads to $\eta'(\xi) >0 $ and $\eta(\xi)<1,$ for every $\xi\in(-\infty,\tau)$. This proves~$(ii)$.

To prove~$(iii)$, we observe that $1 - \eta(\xi) - \beta(\xi)<0$, for every $\xi\in(-\infty,\tau)$, thanks to $(A_2)$, $(A_3)$, \eqref{beta'}, and~$(i)$. Passing to the limit as $\xi\to\tau^-$, we obtain that $\eta(\tau)\geq 1$. Since $\eta'(\xi) \geq0 $ and $\eta(\xi)\in[0,1]$ for every $\xi\in\mathbb{R}$, then we have $\eta(\xi)=1$ for every $\xi\geq\tau.$ As a consequence, $\eta'(\xi)=0$ for every $\xi\geq\tau,$ which implies $\beta(\xi)=0$ for every $\xi\geq\tau$, and using~\eqref{c1}, we have $c=\int_{-\infty}^{\tau} f(\eta(\xi),\beta(\xi))\, \mathrm{d}\xi$. 

At last, $(iv)$ follows from $(A_2)$, $(A_3)$, $(i)$, $(iii)$, and~\eqref{eq-aux1}.
\end{proof}

\begin{remark}
The set $\{\xi\in\mathbb{R}\colon \beta(\xi)>0\}=  \{\xi \in \mathbb R\colon \eta(\xi) <1 \} $ is a non-empty real interval. We define
\begin{equation}\label{eq-tau}
\tau:=\sup\{\xi\in\mathbb{R}\colon \beta(\xi)>0\}.
\end{equation}
If $ \tau \in \mathbb R, $ Lemma~\ref{lemma1} implies that $ (\eta,\beta) \equiv (1,0) $ in $ [\tau,+\infty). $ Thus, by Proposition~\ref{prop1}, $ \beta \in C^2(\mathbb R \setminus \{\tau\}). $
\end{remark}

\begin{remark}
In Lemma~\ref{lemma1}, we demonstrated that any solution of \eqref{eq-sys}--\eqref{eq-bc} satisfies $\eta(\xi) + \beta(\xi) \geq 1$ for all $\xi \in \mathbb{R}$. We stress that in our framework, equality cannot hold in the whole real line; i.e., there is no conservation of total mass. Indeed, under the given boundary conditions, system \eqref{eq-sys} reduces to $g(1-\beta(\xi))h(\beta(\xi)) f(1-\beta(\xi),\beta(\xi)) = 0$ for all $\xi \in \mathbb{R}$. 
Assumptions~$(A_2)$ and~$(A_3)$ would then yield $ \beta(\xi) = k $ for every $ \xi \in \mathbb R, $ with $ k \in \{0, 1\}, $ in contradiction with the boundary conditions.
\end{remark}

\begin{samepage}
In the next lemma, we investigate the properties of the derivatives of $ \eta $ and $ \beta. $
\begin{lemma} \label{lemma1bis}
If $(\eta,\beta)$ is a traveling wave of~\eqref{eq-sys}--\eqref{eq-bc}, then
\begin{enumerate}[nosep,wide=0pt,  labelwidth=20pt, align=left]
\item[$(i)$] $\eta'(-\infty)=0$ and $\beta'(-\infty)=0$;
\item[$(ii)$] $ \eta'(\tau)=0; $
\item[$(iii)$] if $ \tau = +\infty, \beta'(+\infty)=0$;
\item[$(iv)$] if $ \tau \in \mathbb R$, $\lim_{\xi\to\tau^-}h(\beta(\xi))\beta'(\xi)= 0.$
\end{enumerate}
\end{lemma}
\end{samepage}

\begin{proof}
To prove the properties of $ \eta' $ it is sufficient to observe that, from $(A_1)$, \eqref{eq-ode-1}, and~\eqref{eq-bc1}, it follows that
\[ \lim_{ \xi \to -\infty} \eta'(\xi) = \frac 1 c \lim_{\xi \to -\infty} f(\eta(\xi),\beta(\xi))= f(0,1)=0, \]
and similarly
\[ \lim_{\xi \to \tau^-} \eta'(\xi)= f(1,0)=0. \]
Let us now prove the properties of $ \beta' $. First of all notice that $\beta'(-\infty)=0 $ or $\lim_{\xi \to -\infty} \beta'(\xi) $ does not exist, because $ \beta $ is bounded.  Assume by contradiction that $\lim_{\xi \to -\infty} \beta'(\xi)$ does not exist, then, from~\eqref{eq-aux1}, we have also that the limit of the following $C^1$ function
\[
\kappa(\xi)=\frac{1-\beta(\xi)-\eta(\xi)}{g(\eta(\xi))}
\]
does not exist, since
\begin{equation}\label{eq-b3}
\beta'(\xi)=\frac{c\kappa(\xi)}{h(\beta(\xi))}.
\end{equation} 
Hence, we assume that
\[
\ell_1:=\liminf_{\xi\to-\infty} \kappa(\xi)<\limsup_{\xi\to-\infty} \kappa(\xi)=:\ell_2.
\]
Let $\ell\in(\ell_1,\ell_2)$. Consider two sequences $\{\xi_n \}_n$ and $\{\sigma_n \}_n$ such that $\xi_n\to-\infty$ and $\sigma_n\to-\infty$ with $\lim_{n\to+\infty} \kappa(\xi_n)=\ell=\lim_{n\to+\infty} \kappa(\sigma_n)$, $\kappa'(\xi_n)>0$, and $\kappa'(\sigma_n)<0$, for all $n\in\mathbb{N}$. Thanks to the following computation
\[
\kappa'(\xi)=\frac{-\eta'(\xi)-\beta'(\xi)-\kappa(\xi)g'(\eta(\xi))\eta'(\xi))}{g(\eta(\xi))},
\]
and $(A_2)$, we obtain
\[\begin{split}
&-\eta'(\xi_n)-\beta'(\xi_n)-\kappa(\xi_n)g'(\eta(\xi_n))\eta'(\xi_n))>0, \quad \forall n\in\mathbb{N},\\
&-\eta'(\sigma_n)-\beta'(\sigma_n)-\kappa(\sigma_n)g'(\eta(\sigma_n))\eta'(\sigma_n))<0, \quad \forall n\in\mathbb{N}.
\end{split}\]
By passing to the limit as $n\to+\infty$ and using~\eqref{eq-b3}, we deduce that
\[\begin{split}
&\lim_{n \to +\infty} -\eta'(\xi_n)-\beta'(\xi_n)-\kappa(\xi_n)g'(\eta(\xi_n))\eta'(\xi_n)= -\frac{c}{h(1)}\ell\geq0,\\
&\lim_{n \to +\infty} -\eta'(\sigma_n)-\beta'(\sigma_n)-\kappa(\sigma_n)g'(\eta(\sigma_n))\eta'(\sigma_n)= -\frac{c}{h(1)}\ell\leq0.
\end{split}\]
Thus, it follows $\ell=0$, which contradicts the arbitrariness of $\ell$. Hence, $\lim_{\xi\to-\infty} \kappa(\xi)$ exists and, therefore,
\[
\lim_{\xi\to -\infty} \beta'(\xi) =  \lim_{\xi \to -\infty} \frac{c\kappa(\xi)}{h(\beta(\xi))} 
\]
exist, and so we deduce $\beta'(-\infty)=0$. This proves~$(i)$. 

Arguing in a similar way, we can prove that the $\lim_{\xi \to \tau^-} \beta'(\xi) $ exists. If $\tau=+\infty$, since $\beta$ is bounded, we deduce $\beta'(+\infty)=0$, proving $(iii)$. Otherwise, suppose $\lim_{\xi\to \tau^-} \beta'(\xi)<0$. Then, by applying De L'Hopital rule and using~$(ii)$, we obtain
\[
\lim_{\xi\to \tau^-} \beta'(\xi) =   - \frac{c}{g(1)\dot h(0)} \lim_{\xi \to\tau^-} \left(\frac{\eta'(\xi)}{\beta'(\xi)}+1\right)=- \frac{c}{g(1)\dot h(0)}.
\]
From this and~$(A_3)$, we deduce $\lim_{\xi\to\tau^-}h(\beta(\xi))\beta'(\xi)= 0$, proving $(iv)$. This concludes the proof.
\end{proof}

From the proof of Lemma~\ref{lemma1bis}, the next corollary follows straightforwardly.
\begin{corollary}\label{cor-alt}
If $(\eta,\beta)$ is a traveling wave of~\eqref{eq-sys}--\eqref{eq-bc}, then either
\begin{equation}\label{eq-alt}
\lim_{\xi\to\tau^-}\beta'(\xi)=0 \quad\text{or}\quad \lim_{\xi\to\tau^-}\beta'(\xi)=-\frac{c}{g(1)\dot h(0)}.
\end{equation}
\end{corollary}

Lemma~\ref{lemma1} implies that for a traveling wave $(\eta,\beta)$, both $\eta$ and $\beta$ are non-constant and monotonic functions. As a result, it is reasonable to focus on a particular class of traveling waves, namely wavefronts, which meet the following definition.

\begin{definition}[Wavefront and semi-wavefront]\label{def-wavefront}
Let $(\eta,\beta)$ be a traveling wave of~\eqref{eq-sys}. If $\eta$ and $\beta$ are non-constant monotone functions defined on the real line $\mathbb{R}$, then $(\eta,\beta)$ is called a wavefront solution (or simply a wavefront) of system~\eqref{eq-sys}. If $\eta$ and $\beta$ are non-constant monotone functions defined on the half-line $J=(-\infty,\xi_0)$ (respectively, $J=(\xi_0,+\infty)$) for some $\xi_0\in\mathbb{R}$, then $(\eta,\beta)$ is called a semi-wavefront in $J$.
\end{definition}

\begin{remark}\label{rem-sharp}
If $\tau=+\infty$, then $(\eta,\beta)$ is a classical wavefront thanks to Lemma~\ref{lemma1bis}~$(iii)$. Otherwise, if $\tau\in\mathbb{R}$, then from Corollary~\ref{cor-alt}, we can deduce that $(\eta,\beta)$ is a classical wavefront if $\lim_{\xi\to\tau^-}\beta'(\xi)=0$ and it is sharp at $0$ if $\lim_{\xi\to\tau^-}\beta'(\xi)=-{c}/{g(1)\dot h(0)}$. In Section~\ref{section-4}, we will prove that if $\tau\in\mathbb{R}$ then the wavefront is actually sharp at $0$. We also observe that, by Lemma~\ref{lemma1bis}~$(i)$ and~$(ii)$, there are no wavefront profiles $(\eta,\beta)$ that exhibit sharp behavior at~$1$.
\end{remark}

Taking into account both the degeneracy of diffusivity at zero and Lemma~\ref{lemma1}, it is possible for wavefronts to exhibit either classical or sharp behavior, with the latter characterized by the vanishing of $\beta$ in a finite time and in a non-differentiable way (see Remark~\ref{rem-sharp}). Figure~\ref{fig-1} depicts the two admissible wavefront types.

\begin{figure}[h!]
\centering
\begin{tikzpicture}
\begin{axis}[legend style={at={(axis cs:72,30)},anchor=north west},
  tick label style={font=\scriptsize},
  axis y line=middle, 
  axis x line=middle,
  ytick={0,1},
  yticklabel style={anchor=south east},
  xtick={0,7},
  xticklabels={$0$,$\tau$}, 
  xlabel={\small $\xi$}, ylabel={},
every axis x label/.style={
    at={(ticklabel* cs:1.0)},
    anchor=west,
},
every axis y label/.style={
    at={(ticklabel* cs:5.0)},
    anchor=south west
},
  set layers,
  width=8cm,
  height=4cm,
  xmin=-12,
  xmax=12,
  ymin=-0.1,
  ymax=1.2]
\addplot [draw=black, line width=0.6pt, smooth, on layer=axis background]coordinates {(-12,1)(12,1)}; 
\addplot [draw=black, dotted, line width=0.3pt, smooth, on layer=axis background]coordinates {(7,0)(7,1)}; 
\addplot [draw=black, line width=1pt, smooth]coordinates {(12,0) (7,0)};
\addplot [draw=black, line width=1pt, smooth]coordinates {(-12.0000,0.9700) (-11.3540,0.9704) (-10.7280,0.9707) (-10.1220,0.9707) (-9.5344,0.9706) (-8.9660,0.9703) (-8.4160,0.9698) (-7.8838,0.9691) (-7.3691,0.9681) (-6.8715,0.9668) (-6.3904,0.9653) (-5.9255,0.9635) (-5.4763,0.9613) (-5.0424,0.9588) (-4.6234,0.9560) (-4.2188,0.9528) (-3.8281,0.9492) (-3.4510,0.9452) (-3.0871,0.9408) (-2.7358,0.9360) (-2.3967,0.9307) (-2.0695,0.9250) (-1.7536,0.9188) (-1.4487,0.9120) (-1.1543,0.9048) (-0.8700,0.8970) (-0.5952,0.8887) (-0.3297,0.8798) (-0.0730,0.8703) (0.1754,0.8602) (0.4159,0.8495) (0.6490,0.8382) (0.8750,0.8263) (1.0944,0.8136) (1.3077,0.8003) (1.5153,0.7863) (1.7175,0.7715) (1.9149,0.7560) (2.1079,0.7398) (2.4824,0.7051) (2.6648,0.6865) (2.8445,0.6671) (3.1975,0.6258) (3.3718,0.6039) (3.5451,0.5810) (3.8906,0.5327) (4.0637,0.5071) (4.2376,0.4806) (4.5896,0.4246) (5.3223,0.3005) (5.5140,0.2668) (5.7099,0.2321) (6.1165,0.1593) (6.3279,0.1212) (6.5453,0.0820) (6.7692,0.0416) (7.0000,0.0000)};
\addplot [draw=black, line width=1pt, smooth]coordinates {(12,1) (7,1)};
\addplot [draw=black, line width=1pt, smooth]coordinates {(7.0000,1.0000) (6.6366,0.9993) (6.2956,0.9972) (5.9761,0.9937) (5.6771,0.9889) (5.3976,0.9829) (5.1367,0.9757) (4.8933,0.9672) (4.6664,0.9577) (4.4551,0.9471) (4.2585,0.9354) (4.0754,0.9227) (3.9049,0.9091) (3.7461,0.8946) (3.5979,0.8793) (3.4594,0.8631) (3.3295,0.8462) (3.2074,0.8285) (3.0919,0.8102) (2.8773,0.7717) (2.7760,0.7516) (2.6776,0.7310) (2.4851,0.6886) (2.3890,0.6668) (2.2918,0.6447) (2.0899,0.5998) (1.9833,0.5770) (1.8715,0.5541) (1.6288,0.5081) (1.4958,0.4851) (1.3537,0.4621) (1.2017,0.4393) (1.0386,0.4165) (0.8635,0.3940) (0.6754,0.3717) (0.2563,0.3279) (0.0234,0.3065) (-0.2265,0.2856) (-0.4943,0.2651) (-0.7810,0.2451) (-1.0875,0.2256) (-1.4150,0.2068) (-1.7643,0.1886) (-2.1364,0.1710) (-2.5324,0.1542) (-2.9531,0.1382) (-3.3997,0.1229) (-3.8730,0.1086) (-4.3741,0.0951) (-4.9039,0.0826) (-5.4635,0.0711) (-6.0538,0.0607) (-6.6758,0.0513) (-7.3304,0.0431) (-8.0188,0.0360) (-8.7418,0.0302) (-9.5004,0.0256) (-10.2960,0.0224) (-11.1290,0.0205) (-12.0000,0.0200)};
\node at (axis cs:-8,0.15) {\footnotesize{$\eta$}};
\node at (axis cs:-8,0.85) {\footnotesize{$\beta$}};
\end{axis}
\end{tikzpicture}
\qquad
\begin{tikzpicture}
\begin{axis}[legend style={at={(axis cs:72,30)},anchor=north west},
  tick label style={font=\scriptsize},
  axis y line=middle, 
  axis x line=middle,
  ytick={0,1},
  yticklabel style={anchor=south east},
  xtick={0,12},
  xtick style={draw=none},
  xticklabels={0,\color{white}$\tau$}, 
  xlabel={\small $\xi$}, ylabel={},
every axis x label/.style={
    at={(ticklabel* cs:1.0)},
    anchor=west,
},
every axis y label/.style={
    at={(ticklabel* cs:5.0)},
    anchor=south west
},
  set layers,
  width=8cm,
  height=4cm,
  xmin=-12,
  xmax=12,
  ymin=-0.1,
  ymax=1.2]
\addplot [draw=black, line width=0.6pt, smooth, on layer=axis background]coordinates {(-12,1)(12,1)}; 
\addplot [draw=black, line width=1pt, smooth]coordinates {(-12.0000,0.9700) (-11.0810,0.9705) (-10.2000,0.9697) (-9.3546,0.9674) (-8.5444,0.9638) (-7.7685,0.9590) (-7.0259,0.9529) (-6.3157,0.9456) (-5.6367,0.9371) (-4.9882,0.9276) (-4.3691,0.9170) (-3.7785,0.9054) (-3.2153,0.8928) (-2.6788,0.8793) (-2.1678,0.8649) (-1.6814,0.8497) (-1.2188,0.8337) (-0.7788,0.8169) (-0.3605,0.7995) (0.0369,0.7813) (0.4146,0.7626) (0.7733,0.7434) (1.1142,0.7236) (1.7461,0.6826) (2.0390,0.6615) (2.3179,0.6401) (2.8374,0.5964) (3.0799,0.5743) (3.3122,0.5519) (3.7500,0.5070) (3.9574,0.4845) (4.1585,0.4620) (4.5454,0.4172) (5.2852,0.3299) (5.4686,0.3088) (5.6534,0.2881) (6.0308,0.2480) (6.2253,0.2287) (6.4250,0.2100) (6.8438,0.1743) (7.0647,0.1576) (7.2947,0.1415) (7.5347,0.1263) (7.7856,0.1118) (8.0485,0.0983) (8.3242,0.0857) (8.6137,0.0740) (8.9180,0.0634) (9.2380,0.0538) (9.5748,0.0453) (9.9292,0.0379) (10.3020,0.0318) (10.6950,0.0269) (11.1080,0.0232) (11.5430,0.0209) (12.0000,0.0200)};
\addplot [draw=black, line width=1pt, smooth]coordinates {(12.0000,0.9700) (11.4010,0.9689) (10.8240,0.9665) (10.2680,0.9628) (9.7310,0.9580) (9.2135,0.9520) (8.7145,0.9449) (8.2335,0.9368) (7.7695,0.9275) (7.3221,0.9174) (6.8904,0.9062) (6.4739,0.8941) (6.0718,0.8812) (5.6834,0.8675) (5.3080,0.8529) (4.5938,0.8216) (4.2534,0.8050) (3.9234,0.7877) (3.2915,0.7514) (2.9883,0.7325) (2.6926,0.7131) (2.1211,0.6732) (1.8439,0.6527) (1.5716,0.6319) (1.0386,0.5896) (0.0000,0.5031) (-0.2579,0.4814) (-0.5165,0.4597) (-1.0386,0.4165) (-1.3034,0.3951) (-1.5716,0.3740) (-2.1211,0.3324) (-2.4037,0.3120) (-2.6926,0.2920) (-3.2915,0.2532) (-3.6030,0.2346) (-3.9234,0.2164) (-4.5938,0.1818) (-4.9451,0.1654) (-5.3080,0.1498) (-5.6834,0.1348) (-6.0718,0.1206) (-6.4739,0.1073) (-6.8904,0.0947) (-7.3221,0.0831) (-7.7695,0.0724) (-8.2335,0.0626) (-8.7145,0.0539) (-9.2135,0.0462) (-9.7310,0.0396) (-10.2680,0.0342) (-10.8240,0.0299) (-11.4010,0.0268) (-12.0000,0.0250)};
\node at (axis cs:-8,0.2) {\footnotesize{$\eta$}};
\node at (axis cs:-8,0.85) {\footnotesize{$\beta$}};
\end{axis}
\end{tikzpicture}
\caption{Admissible wavefront behaviors $(\eta,\beta)$ of~\eqref{eq-sys}: depicting a classical profile on the left and a sharp profile on the right.}
\label{fig-1}
\end{figure}
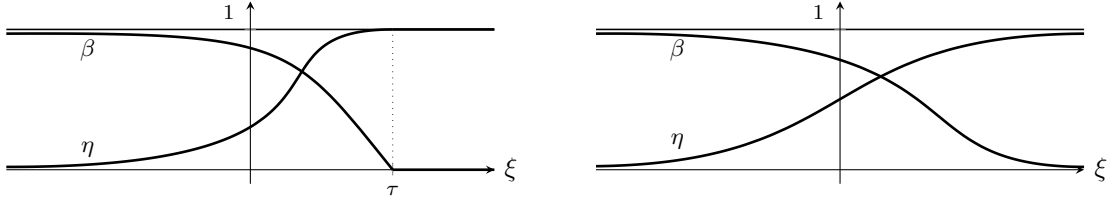

\begin{proposition}\label{prop28}
If $(\eta,\beta)$ is a wavefront of~\eqref{eq-sys}--\eqref{eq-bc}, then
\begin{equation}\label{est-c}
c\geq c_\sharp:=\max \left\{ \sqrt{L_1M_g \int_0^1 g(1-r)h(r)r\,\mathrm{d}r}, \sqrt{2L_1M_g \int_0^1 (1-r)g(1-r)h(r)r\,\mathrm{d}r} \right\}.
\end{equation}
\end{proposition}

\begin{proof}
We will prove two different estimates as follows.

\noindent
\emph{First estimate.}
From assumption~$(A_1)$, \eqref{eq-aux1}, and~\eqref{c1}, we have 
\[\begin{split}
c=\int_{-\infty}^{+\infty}f\left(\eta(\xi),\beta(\xi)\right)\,\mathrm{d}\xi&\geq L_1 \int_{-\infty}^{+\infty}\eta(\xi)\beta(\xi)\,\mathrm{d}\xi\\
&= L_1  \int_{-\infty}^{+\infty}\beta(\xi)\frac{c-c\beta(\xi)-g\left(\eta(\xi)\right)h\left(\beta(\xi)\right)\beta'(\xi)}{c}\,\mathrm{d}\xi\\
&\geq L_1  \int_{-\infty}^{+\infty}\beta(\xi)\frac{-g\left(\eta(\xi)\right)h\left(\beta(\xi)\right)\beta'(\xi)}{c}\,\mathrm{d}\xi.
\end{split}\]
Using assumption~$(A_2)$ and Lemma~\ref{lemma1}, it thus follows
\[
c^2\geq L_1 M_g \int_{-\infty}^{+\infty}-g\left(1-\beta(\xi)\right)h\left(\beta(\xi)\right)\beta(\xi)\beta'(\xi)\,\mathrm{d}\xi.
\]
Applying the integration by substitution formula and~\eqref{eq-bc}, we finally get
\begin{equation}\label{est-c1}
c\geq \sqrt{L_1M_g \int_0^1 g(1-r)h(r)r\,\mathrm{d}r}.
\end{equation} 

\noindent \emph{Second estimate.}
Inspired by~\cite{BNS-85,Ma-85}, let us proceed in a different way. Multiplying \eqref{eq-ode-2} by $ \beta $  and integrating in $(-\infty,+\infty),$ according to Lemma~\ref{lemma1} and~\eqref{eq-bc}, we obtain
\[\begin{split}
0&= \int_{-\infty}^{+\infty} \left( \left(g(\eta(\xi))h(\beta(\xi))\beta'(\xi) \right)'	+c\beta'(\xi) + f(\eta(\xi),\beta(\xi))	\right)\beta(\xi) \, \mathrm{d}\xi \\
&= \left[g(\eta(\xi))h(\beta(\xi))\beta'(\xi) \beta(\xi)\right]_{-\infty}^{+\infty} - \int_{-\infty}^{+\infty}  g(\eta(\xi))h(\beta(\xi))\beta'(\xi)^2 \, \mathrm{d}\xi +\left[ \frac c 2 \beta(\xi)^2 \right]_{-\infty}^{+\infty}\\
&\quad+ \int_{-\infty}^{+\infty} f(\eta(\xi),\beta(\xi)) \beta(\xi) \, \mathrm{d}\xi  \\
&=   -  \int_{-\infty}^{+\infty}  g(\eta(\xi))h(\beta(\xi))\beta'(\xi)^2 \, \mathrm{d}\xi  - \frac c 2 + \int_{-\infty}^{+\infty} f(\eta(\xi),\beta(\xi)) \beta(\xi) \, \mathrm{d}\xi \\ 
&\leq   -  \int_{-\infty}^{+\infty}  g(\eta(\xi))h(\beta(\xi))\beta'(\xi)^2 \, \mathrm{d}\xi  - \frac c 2 +c \\
&=    -  \int_{-\infty}^{+\infty}  g(\eta(\xi))h(\beta(\xi))\beta'(\xi)^2 \, \mathrm{d}\xi  + \frac c 2. 
\end{split}\]
This leads to 
\begin{equation}\label{c2}
\frac{c}{2}\geq\int_{-\infty}^{+\infty}  g(\eta(\xi))h(\beta(\xi))\beta'(\xi)^2\, \mathrm{d}\xi.   
\end{equation}
On the other hand, multiplying~\eqref{eq-ode-2} by $g(\eta)h(\beta)\beta'$, integrating in~$(-\infty,+\infty)$, and using Lemma~\ref{lemma1} and~\eqref{eq-bc}, we have
\[\begin{split}
0&=  \int_{-\infty}^{+\infty} \left( \left(g(\eta(\xi))h(\beta(\xi))\beta'(\xi) \right)'	+c\beta'(\xi) + f(\eta(\xi),\beta(\xi))	\right) g(\eta(\xi))h(\beta(\xi)) \beta'(\xi) \, \mathrm{d}\xi \\
&=  \left[ \frac{1}{2} \left(g(\eta(\xi))h(\beta(\xi))\beta'(\xi)\right)^2 \right]_{-\infty}^{+\infty} + c \int_{-\infty}^{+\infty} g(\eta(\xi))h(\beta(\xi))\beta'(\xi)^2 \, \mathrm{d}\xi \\
&\quad+ \int_{-\infty}^{+\infty} f(\eta(\xi),\beta(\xi)) g(\eta(\xi))h(\beta(\xi))\beta'(\xi) \, \mathrm{d}\xi \\
&=   c \int_{-\infty}^{+\infty} g(\eta(\xi))h(\beta(\xi))\beta'(\xi)^2  \mathrm{d}\xi + \int_{-\infty}^{+\infty} f(\eta(\xi),\beta(\xi)) g(\eta(\xi))h(\beta(\xi))\beta'(\xi) \, \mathrm{d}\xi.
\end{split}\]
Using~\eqref{c2}, we deduce
\begin{equation} \label{c3}
\frac{c^2}{2}\geq- \int_{-\infty}^{+\infty} f(\eta(\xi),\beta(\xi)) g(\eta(\xi))h(\beta(\xi))\beta'(\xi) \, \mathrm{d}\xi. 
\end{equation}
Thanks to Lemma~\ref{lemma1}, and using~$(A_1)$ and $(A_2)$, we have
\[\begin{split}
f(\eta(\xi),\beta(\xi))&\geq L_1\eta(\xi)\beta(\xi)\geq L_1 \left(1-\beta(\xi)\right)\beta(\xi),	\quad \forall\xi\in\mathbb{R},\\
g(\eta(\xi))&\geq M_g g(1-\beta(\xi)), \quad \forall\xi\in\mathbb{R}.
\end{split}\]
From~\eqref{c3}, it thus follows 
\[
\frac{c^2}{2}\geq- L_1M_g\int_{-\infty}^{+\infty} \left(1-\beta(\xi)\right)\beta(\xi) g(1-\beta(\xi)) h(\beta(\xi)) \beta'(\xi) \, \mathrm{d}\xi.
\]
Applying the integration by substitution formula and~\eqref{eq-bc}, we obtain
\begin{equation}\label{est-c}
c\geq \sqrt{2 L_1M_g \int_0^1 (1-r) r g(1-r)h(r)\,\mathrm{d}r}.
\end{equation} 
Thus, \eqref{est-c1} and \eqref{est-c} yields the thesis.
\end{proof}

\begin{remark}\label{rem-csharp}
Under assumption $(A_2)$ and $(A_3)$, we recover many cases studied in the literature, where the diffusivity is expressed taking $g(s)=s^{\alpha}$ and $h(r)=r^{\gamma}$, with $ \alpha , \gamma > 0. $ In particular this implies that $ M_g =1. $ Even in this simple case it is not possible to establish which of the two estimates is better than the other one, i.e. which of the two lower bounds is the biggest one. Indeed, take $ \gamma = 1. $ Then
\[ \sqrt{L_1 M_g\int_0^1 g(1-r)h(r)r\,\mathrm{d}r} = \sqrt{L_1\int_0^1 (1-r)^{\alpha} r^2 \,\mathrm{d}r} = \sqrt{\frac {2L_1} {(\alpha +1)(\alpha + 2) (\alpha + 3)}} \]
while
\[\sqrt{2L_1M_g \int_0^1 (1-r)g(1-r)h(r)r\,\mathrm{d}r} =  \sqrt{2L_1\int_0^1 (1-r)^{\alpha + 1} r^2 \,\mathrm{d}r} = \sqrt{\frac {4L_1} {(\alpha +2)(\alpha + 3) (\alpha + 4)}}. \]
The first lower bound is bigger than the second one if and only if 
\[ \frac 1 {\alpha + 1} > \frac {2} {\alpha + 4}, \]
i.e., if and only if $ \alpha < 2. $

In particular, when $ \alpha = \gamma =1 $ and $ f(s,r) =sr, $ we recover the case studied by~\cite{SMGA-01} where an estimate for the lower bound for the speed is $c\geq\sqrt{\frac{1}{6}}$. They also guess that the threshold speed is $\sqrt{\frac{1}{2}}$. We notice that our best estimate in this case is the first one, i.e.
$c_\sharp= \sqrt{\int_0^1 (1-r) r^2\,\mathrm{d}r}= \sqrt{\frac{1}{12}}.$
\end{remark}

%-_-_-_-_-_-_-_-_-_-_-_-_-_-_-_-_-_-_-_-_-_-_-_-_-_-_-_-_-_-_-_-_-_-_
\section{Existence and uniqueness of wavefronts}\label{section-3}
%-_-_-_-_-_-_-_-_-_-_-_-_-_-_-_-_-_-_-_-_-_-_-_-_-_-_-_-_-_-_-_-_-_-_
In this section, our primary objective is two-fold. Firstly, we establish the existence of a unique semi-wavefront, $(\eta,\beta)$ of~\eqref{eq-sys}--\eqref{eq-bc1} on the negative half-line, for every speed $c > 0$. This existence and uniqueness, up to horizontal shifts, are detailed in Theorem \ref{th-m} and Theorem \ref{th-uniq}, respectively. Secondly, we prove the existence of a constant, $\rho_*>0$, such that for all $c \geq \rho_*$, we can extend the solution $(\eta,\beta)$ into a classical wavefront~\eqref{eq-sys}--\eqref{eq-bc} on the all real line. This result is contained in Theorem \ref{th-positive}.

In terms of assumptions $(A_1)$--$(A_3)$, apart from the regularity conditions in the proof of the uniqueness of the semi-wavefront, one can assume that $f$ is a Lipschitz continuous function.

%------------------------------------------------------------------------
\subsection{Existence of semi-wavefronts on the negative half-line}\label{sub-ex-n}
%------------------------------------------------------------------------

Given a wavefront $(\eta,\beta)$, we define the following quantities:
\[
\eta_0:=\eta(0), \qquad \beta_0:=\beta(0), \qquad \beta'_0:=\beta'(0).
\]

\begin{proposition}\label{proplus}
If $(\eta,\beta)$ is a wavefront of~\eqref{eq-sys}--\eqref{eq-bc}, then
\begin{equation} \label{etalus}
\eta_0 e^{\frac{L_2\xi}{c}}\leq \eta(\xi)\leq \eta_0 e^{\frac{L_1(1-\eta_0) \xi}{c}},	\quad\forall\xi\leq 0.
\end{equation}
\end{proposition}

\begin{proof}
From~\eqref{eq-ode-1} and using $(A_1)$ and Lemma~\ref{lemma1}, for every $\xi\in(-\infty,0]$, we can deduce that
\[ \eta' (\xi) = \frac{f(\eta(\xi),\beta(\xi))}c \leq \frac {L_2 \eta(\xi)\beta(\xi)} c \leq \frac {L_2\eta(\xi)}c .\]
Thus, by Gronwall's lemma, we get that
\[\eta(\xi) \geq \eta_0 e^{\frac{L_2 \xi} c}, \quad \forall \xi\in(-\infty,0].
\]
Similarly, using $(A_1)$ and Lemma~\ref{lemma1}, for every $\xi\in(-\infty,0]$, we have
\[
\eta'(\xi) = \frac{f(\eta(\xi),\beta(\xi))}c \geq \frac {L_1 \eta(\xi)\beta(\xi)} c \geq \frac {L_1 \beta_0\eta(\xi)}c \geq \frac {L_1 (1-\eta_0)\eta(\xi)} c.
\]
Applying Gronwall's lemma, we have
\[\eta(\xi) \leq \eta_0 e^{\frac{L_1(1-\eta_0) \xi} c},\quad \forall\xi\in(-\infty,0]. \qedhere
 \]
\end{proof}

\begin{remark}\label{rem32}
Notice that condition \eqref{etalus} is satisfied by any solution $ \eta $ of \eqref{eq-ode-1} and any function $ \beta $ such that $ \beta_0 \leq \beta(\xi) \leq 1 $ for every $\xi\leq 0.$ 
\end{remark}

Let us fix $c>0$ and $\eta_0\in(0,1)$. We introduce
\[
\mathfrak{N}:=\left\{\eta\in C^1(-\infty,0] \colon \eta_0 e^{\frac{L_2\xi}{c}}\leq \eta(\xi)\leq \eta_0 e^{\frac{L_1(1-\eta_0) \xi}{c}},\, \eta'(\xi)\geq0,\, \forall \xi\in(-\infty,0]\right\}.
\]
We observe that $\mathfrak{N}$ is a closed and convex subset of the separable Fr\'echet space $C^1(-\infty,0]$.

\begin{proposition} \label{th-beta}
For every $\eta\in\mathfrak{N}$ and $c>0$ there exists a unique $\beta_0\in(0,1)$ such that the following problem
\begin{equation}\label{eq-beta}
\begin{cases}
\beta'=\dfrac{c\left(1-\beta-\eta\right)}{g(\eta)h(\beta)},\\
\beta(0)=\beta_0,
\end{cases}\end{equation}
has a solution $\beta$ satisfying $\beta(\xi)\in[\beta_0,1)$, $\beta'(\xi)<0 $ for all $\xi\in(-\infty,0]$, and $\beta(-\infty)=1.$
\end{proposition}

\begin{proof}
Let us fix $\beta_0\in(0,1)$ and consider the solution $\beta$ of~\eqref{eq-beta} defined in its maximal interval of existence. We define 
\[\begin{split}
&\mathcal{A}:=\{\beta_0\in(0,1)\colon \exists\xi_0<0\text{ s.t. }\beta(\xi_0)=1\},\\
&\mathcal{B}:=\left\{\beta_0\in(0,1)\colon \exists\xi_0<0\text{ s.t. }\mathrm{lim}_{\xi\to\xi_0^+}\beta(\xi)=0\right\}.
\end{split}\]

\noindent\emph{Step $0$. If there is $\xi_1\leq0$ such that $\beta'(\xi_1)>0$, then $\beta_0\in\mathcal{B}$.} Consider the function $ \Gamma(\xi)= 1 - \beta(\xi) - \eta(\xi)$, then 
\[
\beta'(\xi)=\frac{c\Gamma(\xi)}{g(\eta(\xi))h(\beta(\xi))}.
\]
Since $\beta'(\xi_1)>0$, then we obtain $\Gamma(\xi_1)>0$. Let us show that $\beta'(\xi)>0$, for every $\xi$ in the maximal interval of existence of $\beta$. Suppose by contradiction that there exists $\xi_0<\xi_1$ such that $\beta'(\xi_0)=0 $ and $ \beta'(\xi) > 0 $ for every $ \xi \in (\xi_0, \xi_1). $  Then, by definition of~$ \mathfrak{N},  \Gamma'(\xi) < 0 $ for every $ \xi \in (\xi_0, \xi_1).$ Thus, being $\Gamma$ decreasing in $(\xi_0,\xi_1)$, we have
\[0= \beta'(\xi_0) = \frac {c\Gamma (\xi_0)}{g(\eta(\xi_0)) h(\beta(\xi_0))} > \frac  {c\Gamma (\xi_1)}{g(\eta(\xi_0)) h(\beta(\xi_0))} >0,
\]
which is a contradiction. Moreover, by the same reasoning, for every $\xi$ in the maximal interval of existence of $\beta$, we have
\[\beta'(\xi) > \frac  {c\Gamma (\xi_1)}{\max_{[0,\eta_0]} g \max_{[0,\beta(\xi_1)]}h} > 0. \]
Accordingly, there is $\xi_2<\xi_1$ such that $\beta(\xi_2)=0$ and so $\beta_0 \in\mathcal{B}. $ 

\noindent\emph{Step $1$. $\mathcal{A}$ is a non-empty open interval.} Consider the solution $\beta$ of~\eqref{eq-beta} defined in its maximal interval of existence such that $\beta(0)=1$. Notice that, from~\eqref{eq-beta}, $(A_2)$, and~$(A_3)$, $\beta'(0)<0$. We now claim that $\beta(\xi)>1$ for every $\xi<0$. By contradiction, let $\xi_1<0$ such that $\beta(\xi)>1$ for every $\xi\in(\xi_1,0)$ and $\beta(\xi_1)=1$. Then, since $\eta\in\mathfrak{N}$ and $c>0$, from~\eqref{eq-beta} follows $\beta'(\xi_1)<0$ which is a contradiction. Using the continuous dependence of solutions on initial data, there exists $\delta > 0$ such that for any $\beta_0 \in (1-\delta,1)$, the solution $\beta$ of~\eqref{eq-beta} defined on its maximal existence interval is such that there exists $\xi_0$ such that $\beta(\xi_0)=1$. Thus $\mathcal{A} \not= \emptyset. $ Using the continuous dependence on initial data again, we can prove that $\mathcal{A}$ is open. To prove that $\mathcal{A}$ is an interval, it is sufficient to show that it is connected. Take $ \beta_0^1,\beta_0^2 \in\mathcal{A}$ and $\beta_0^3 $ such that $ \beta_0^1 < \beta_0^3 < \beta_0^2. $ The uniqueness of the solution of the Cauchy problem \eqref{eq-beta} implies that the trajectory of the solution corresponding to $\beta_0^3$ lies between the trajectories corresponding to $ \beta_0^1 $ and $\beta_0^2, $ which proves that $ \beta_0^3 \in\mathcal{A}. $ 

\noindent\emph{Step $2$. $\mathcal{B}$ is a non-empty interval.} 
Take $ \beta_0 < 1 -\eta_0 $ and consider the solution $\beta$ of~\eqref{eq-beta} defined in its maximal interval of existence such that $\beta(0)=\beta_0. $ Then, by \eqref{eq-beta}, $(A_2)$, and $(A_3)$, we get that $ \beta'(0)>0$. From~\textit{Step~$0$}, we deduce $\beta_0 \in\mathcal{B}. $ Using the uniqueness of the solution of the Cauchy problem, it is possible to prove that also $\mathcal{B}$ is connected. 

\noindent\emph{Step $3$. $\mathcal{A}\cap \mathcal{B}=\emptyset$.}
Suppose that there exists $ \beta_0 \in \mathcal{A} \cap \mathcal{B}. $ Denote by $\beta$ the solution of~\eqref{eq-beta} defined in its maximal interval of existence such that $\beta(0) = \beta_0, \xi_1 = \inf\{\xi\colon \beta(\xi)>0 \} $ and by $ \xi_2 = \max \{ \xi\colon\beta(\xi)=1\}. $ Then, according to the definition of $\mathfrak{N}$ and using $(A_2)$, and $(A_3)$, we have   
\[\lim_{\xi \to \xi_1} \beta'(\xi) = + \infty.\]
Therefore, the maximal interval of existence of $\beta$ is $(\xi_1,0),$ i.e. $\xi_2> \xi_1.$ 
Then, by~\eqref{eq-beta}, $(A_2)$, and $(A_3)$, we get that $ \beta'(\xi_2)<0, $ and so $ \beta (\xi)> 1 $, for every $\xi$ in a left neighbourhood of $\xi_2$. Since $ \beta(\xi_1)=0, $ there exists $\xi_3 \in (\xi_1,\xi_2)$ such that $ \beta'(\xi_3) =0 $ and $ \beta(\xi_3)>1, $ which leads to a contradiction reasoning like in \textit{Step~$1$}. 

\noindent\emph{Step $4$. If $\beta(0)=\inf\mathcal{A}$ then $\beta(\xi)<1$, for every $\xi<0$, and $\beta(-\infty)=1$.} Let $\beta$ be the solution such that $\beta(0)=\inf\mathcal{A}$. Then, by \textit{Step~$1$}, we have $\beta(\xi)<1$, for every $\xi<0$. We prove that $\beta'(\xi)\leq 0$ for every $\xi<0$. Suppose by contradiction that there is $\xi_1<0$ such that $\beta'(\xi_1)>0$. Then, from~\textit{Step~$0$}, it follows that $\inf\mathcal{A}\in\mathcal{B}$. Therefore, there is $\xi_2<\xi_1$ such that $\beta(\xi_2)=0<\inf\mathcal{A}$. By the continuous dependence of the solutions on the initial data, there is $\rho$ in a right neighborhood of $\inf\mathcal{A}$ such that 
\[
\beta_\rho(\xi_2)\leq \inf\mathcal{A}<\rho,
\]
where $\beta_\rho$ is the solution of~\eqref{eq-beta} associated with the initial condition $\beta(0)=\rho$. By the mean value theorem, there is $\xi_3>\xi_2$ such that $\beta'_\rho(\xi_3)>0$. From~\textit{Step~$0$} and~\textit{Step~$3$}, $\rho\in\mathcal{B}$ and so $\rho\not\in\mathcal{A}$, which is a contradiction because $\mathcal{A}=(\inf\mathcal{A},1)$. Thus, we obtain that $\beta'(\xi)\leq0$, for every $\xi<0$, that is $1-\beta(\xi)-\eta(\xi)\leq0$ thanks to Lemma~\ref{lemma0}, $(A_2)$, and $(A_3)$. It now follows $1-\eta(\xi)\leq\beta(\xi)<1$ and so, by passing to the limit as $\xi\to-\infty$, it follows $\beta(-\infty)=1$ (according to the definition of $\mathfrak{N}$). Arguing as in Lemma~\ref{lemma1}, we conclude that $\beta'(\xi)<0$, for every $\xi<0$.

\noindent\emph{Step $5$. $\inf\mathcal{A}=\sup\mathcal{B}$.} Suppose by contradiction that $\inf\mathcal{A}>\sup\mathcal{B}$. Let $\beta$ be the solution of~\eqref{eq-beta} associated with the initial condition $\beta(0)=\inf\mathcal{A}$ defined in $(-\infty,0$). For any given $\rho\in(\sup\mathcal{B},\inf\mathcal{A})$, let $\beta_\rho$ be the solution of~\eqref{eq-beta} associated with the initial condition $\beta(0)=\rho$ defined in its maximal interval of existence. Since $\rho\not\in\mathcal{B}$, then $\beta'_\rho(\xi)\leq0$ for every $\xi$ in its maximal interval of existence. Since $\rho\not\in\mathcal{A}$, we can argue as in~\textit{Step~$4$} and deduce that $\beta_\rho$ is defined in $(-\infty,0)$ and $\beta_\rho(-\infty)=1$. By the uniqueness of the solutions of the associated Cauchy problems, we have $\beta_\rho(\xi)<\beta(\xi)$, for every $\xi\leq0$. Hence, by using~\eqref{eq-beta}, we obtain
\[\begin{split}
0>\int^0_{-\infty}\frac{c}{g(\eta(\xi))}\left(\beta_\rho(\xi)-\beta(\xi)\right)\,\mathrm{d}\xi&=\int_{-\infty}^0 \left(h(\beta(\xi))\beta'(\xi)-h(\beta_\rho(\xi))\beta'_\rho(\xi)\right)\,\mathrm{d}\xi\\
&=\int_{\rho}^{\inf\mathcal{A}}h(y)\,\mathrm{d}y>0,
\end{split}\]
which is a contradiction. This concludes the proof.
\end{proof}

\begin{remark}\label{rem-i}
If $\beta_0<1-\eta_0$, then there is $\xi_0\in(-\infty,0)$ such that $\beta(\xi_0)=0$, that is $\beta_0\in \mathcal{B}$.
\end{remark}

Using fixed-point theory in separable Fréchet spaces, we prove what follows. 

\begin{theorem}\label{th-m}
For every $c>0$ and for every $\eta_0\in(0,1)$, there is a semi-wavefront $(\eta,\beta)$ of~\eqref{eq-sys} defined in $(-\infty,0]$ and satisfying~\eqref{eq-bc1}.
\end{theorem}

\begin{proof}
Let us fix $c>0$ and $\eta_0\in(0,1)$. We consider $\eta\in\mathfrak{N}$. Then, we take $\tilde{\beta}$ the solution of~\eqref{eq-beta} satisfying $\tilde{\beta}(-\infty)=1$ (whose existence is guaranteed by Proposition~\ref{th-beta}). We thus introduce the operator $\mathcal{V}\colon\mathfrak{N}\to C^1(-\infty,0]$ such that $\mathcal{V}(\eta)=\tilde{\beta}$.

Next, we consider the problem
\begin{equation}\label{eq-eta}
\begin{cases}
\zeta'=\dfrac{f(\zeta,\tilde{\beta})}{c},\\
\zeta(0)=\eta_0.
\end{cases}\end{equation}
From $(A_1)$, we observe that problem~\eqref{eq-eta} has a unique solution $\tilde{\eta}$. We therefore consider $\mathcal{W}\colon\mathcal{V}(\mathfrak{N})\subseteq C^1(-\infty,0]\to\mathfrak{N}$ such that $\mathcal{W}(\tilde\beta)=\tilde{\eta}$.

In this manner, we define the operator
\[\mathcal{T}\colon\mathfrak{N}\to C^1(-\infty,0], \quad \mathcal{T}=\mathcal{W}\circ\mathcal{V}.\]
%defined as 
%\[
%\eta\mapsto \tilde{\beta} \text{ solution of~\eqref{eq-beta} with }\tilde{\beta}(-\infty)=1 \;\mapsto \tilde{\eta} \text{ solution of~\eqref{eq-eta}}.
%\]
We observe that $ \mathcal T $ is well-defined. Moreover, thanks to Proposition~\ref{th-beta}, $\tilde{\beta}(\xi)\in[\beta_0,1),$ for every $\xi\in(-\infty,0]$. From Proposition~\ref{proplus} and Remark~\ref{rem32}, we deduce that $\mathcal{T}(\mathfrak{N})\subseteq\mathfrak{N}$. 

Next, we show that $\mathcal{T}$ is compact. Let us consider $(\tilde{\eta}_j)_j\subseteq\mathcal{T}(\mathfrak{N}).$ Thus, we notice that $(\tilde{\eta}_j)_j$ is equibounded since $\mathcal{T}(\mathfrak{N})\subseteq\mathfrak{N}$ and $\mathfrak{N}$ is bounded. First, we show that $(\tilde{\eta}'_j)_j$ is equibounded. For every $j$, associated to $\tilde{\eta}_j$ there is $\eta_j\in\mathfrak{N}$ and $\tilde{\beta}_j$ which is the solution of~\eqref{eq-beta} with $\tilde{\beta}_j(-\infty)=1$ such that $\tilde{\eta}_j$ is the solution of~\eqref{eq-eta}. Hence, it holds
\begin{equation}\label{eq-ee2}
\tilde\eta'_j(\xi)=\dfrac{f(\tilde\eta_j(\xi),\tilde{\beta}_j(\xi))}{c}\leq\frac{\max_{[0,\eta_0]\times[0,1]} f}{c}=:K_1, \quad\xi\leq 0,
\end{equation}
and so $(\tilde{\eta}'_j)_j$ is equibounded which in turns implies that $(\tilde{\eta}_j)_j$ is equicontinous.
Second, we prove that $(\tilde\eta'_j)_j$ is equicontinuous on compact subsets of $(-\infty,0]$. To this purpose, let us fix $T>0$ and consider $-T\leq\xi<\xi_1\leq0$. Since $f$ is a Lipschitz function (with constant $L$), we deduce by the Mean Value Theorem
\begin{equation}\label{eq-ee}
\begin{split}
|\tilde\eta'_j(\xi)-\tilde\eta'_j(\xi_1)|&=\frac{\left|f\left(\tilde\eta_j(\xi),\tilde{\beta}_j(\xi)\right) - f\left(\tilde\eta_j(\xi_1),\tilde{\beta}_j(\xi_1)\right)	\right|}{c}\\
&\leq \frac{L\left(\left|	\tilde\eta_j(\xi)-\tilde\eta_j(\xi_1)\right| +\left|\tilde\beta_j(\xi)-\tilde\beta_j(\xi_1)\right|	\right)		}{c}\\
&= \frac{L\left(\left|	\tilde\eta'_j(\xi^\sharp)\right| +\left|\tilde\beta'_j(\xi^\star)\right|	\right)\left|\xi-\xi_1	\right|		}{c},\quad \text{for some }\xi^\sharp,\xi^\star\in(\xi,\xi_1).
\end{split}\end{equation}
Since $\eta_j\in\mathfrak{N}$ and $\xi^\star\in[-T,0]$, we have $\eta_0\geq\eta_j(\xi^\star)\geq \eta_0 e^{\frac{-L_2 T}{c}}=:m$. Moreover, from Remark~\ref{rem-i}, we have $1\geq\tilde\beta_j(\xi^\star)\geq \tilde\beta_j(0)\geq1-\eta_0$. Thanks to~\eqref{eq-beta} and using these inequalities, we obtain
\[
\left|\tilde\beta'_j(\xi^\star)\right|=\frac{c\left(\eta_j(\xi^\star)+\tilde\beta_j (\xi^\star)-1	\right)}{g\left(\eta_j(\xi^\star)\right) h\left((\tilde\beta_j(\xi^\star)\right)}\leq \frac{c\eta_0}{\left(\min_{\left[m,\eta_0\right]}g\right) \left(\min_{\left[1-\eta_0,1\right]}h\right)}=:K_2.
\]
Thank to the previous inequality and~\eqref{eq-ee2}, from~\eqref{eq-ee} we have
\[
|\tilde\eta'_j(\xi)-\tilde\eta'_j(\xi_1)|\leq \frac{L (K_1+K_2)}{c}\left| \xi-\xi_1\right|.
\]
By the Ascoli-Arzel\`a Theorem, we obtain that $(\tilde\eta_j)_j$ has a subsequence converging in $C^1(-\infty,0]$, and so $\mathcal{T}$ is compact. 

At last, we show that $\mathcal{T}$ is sequentially closed.  It is sufficient to prove that the maps $\mathcal{V}$ and $\mathcal{W}$ are sequentially closed. To prove that $\mathcal{V}$ is sequentially closed we consider $(\tilde{\beta}_j)_j\subseteq \mathcal{V}(\mathfrak{N})$. Then, for every $j$, there is $\eta_j\in\mathfrak{N}$ such that $\mathcal{V}(\eta_j)=\tilde{\beta}_j$. If $\eta_j\to\bar\eta$ and $\tilde\beta_j\to\bar\beta$ in $C^1(-\infty,0]$, then we aim to prove that $\bar{\beta}=\mathcal{V}(\bar\eta)$.
Since $\tilde\beta_j$ is solution of~\eqref{eq-beta}, it follows that for every $\xi\in(-\infty,0]$
\[
\tilde\beta_j (\xi)=\tilde\beta_j(0)-\int_\xi^0 \frac{c\left(1-\eta_j(\sigma)-\tilde\beta_j (\sigma)\right)	}{g\left(\eta_j(\sigma)\right) h((\tilde\beta_j(\sigma))	}\,\mathrm{d}\sigma \to
\bar\beta(0)-\int_\xi^0 \frac{c\left(1-\bar\eta(\sigma)-\bar\beta (\sigma)\right)	}{g\left(\bar\eta(\sigma)\right) h((\bar\beta(\sigma))	}\,\mathrm{d}\sigma,
\] 
because the topology in $C^1(-\infty,0]$ is equivalent to the uniform convergence on compact subsets. As a consequence, $\bar\beta=\mathcal{V}(\bar\eta)$.

We are left to prove that $\mathcal{W}$ is sequentially closed. To do so, we consider $(\tilde{\eta}_j)_j\subseteq \mathfrak{N}$. Then, for every $j$, there is $\tilde\beta_j\in \mathcal{V}(\mathfrak{N})$ such that $\mathcal{W}(\tilde\beta_j)=\tilde{\eta}_j$. If $\tilde\beta_j\to\bar\beta$ and $\tilde\eta_j\to\bar\eta$ in $C^1(-\infty,0]$, then we aim to prove that $\bar{\eta}=\mathcal{W}(\bar\beta)$.
Since $\tilde\eta_j$ is solution of~\eqref{eq-eta}, it follows that, for every $\xi\in(-\infty,0]$,
\[
\tilde\eta_j (\xi)=\tilde\eta_j(0)-\int_\xi^0 \frac{f\left(\tilde\eta(\sigma),\tilde\beta(\sigma)\right)}{c}\,\mathrm{d}\sigma \to
\bar\eta(0)-\int_\xi^0 \frac{f\left(\bar\eta(\sigma),\bar\beta(\sigma)\right)}{c}\,\mathrm{d}\sigma.
\] 
We obtain similarly that $\bar{\eta}=\mathcal{W}(\bar\beta)$. Thus, $\mathcal{T}$ is sequentially closed and since it is also compact, it follows that $\mathcal{T}$ is continuous. 

Then, we apply the Schauder-Tychonoff fixed point theorem, and we get a fixed point $\tilde{\eta}$ for~$\mathcal{T}$. Since $\tilde{\eta}\in\mathfrak{N}$, then $\tilde{\eta}(-\infty)=0$ and so we conclude that the pair $(\tilde{\eta},\tilde{\beta})$ is a semi-wavefront of~\eqref{eq-sys} in $(-\infty,0]$ satisfying~\eqref{eq-bc1}. 
\end{proof}

%------------------------------------------------------------------------
\subsection{Uniqueness and properties of semi-wavefronts on the negative half-line}\label{sub-uniq}
%------------------------------------------------------------------------
In this section, we employ a desingularization technique on system~\eqref{eq-sys} similar to one used in~\cite{GaMa-22,SMGA-01}. To ensure uniqueness, we proceed to demonstrate Theorem~\ref{th-uniq} by adopting the methodology outlined in~\cite{AiHuang-07}.

\begin{description}[leftmargin=*]
\item[Reduction to a desingularized first-order problem.]
For every $c>0$ and $\eta_0\in(0,1)$, let $(\eta,\beta)$ be a semi-wavefront on the interval $(-\infty,0]$. We consider the function $y=\Phi(\xi)$ such that 
\[\begin{cases}
\Phi'(\xi)=\dfrac{1}{g(\eta(\xi))h(\beta(\xi))}, \quad \xi\in(-\infty,0],\\
\Phi(0)=0.
\end{cases}\]
Notice that the function $\Phi$ is a diffeomorphism with $\Phi(-\infty)=-\infty$ because, by Remark~\ref{rem-i}, we have
\[
\Phi'(\xi)\geq\dfrac{1}{\max_{[0,\eta_0]}g \max_{[1-\eta_0,1]}h}>0.
\]
Hence, we can define
\[
p(y):=\eta\left(\Phi^{-1}(y)\right),\quad q(y):=\beta\left(\Phi^{-1}(y)\right)+\eta\left(\Phi^{-1}(y)\right)-1.
\]
Then, for every $y\in(-\infty,0]$, $(p,q)$ solves the following first-order problem
\begin{equation}\label{eq-manifold}
\begin{cases}
\dot p(y)=\dfrac{f(p(y),q(y)-p(y)+1) g(p(y)) h(q(y)-p(y)+1)}{c},\\
\dot q(y)=-c q(y)+\dfrac{f(p(y),q(y)-p(y)+1 ) g(p(y)) h(q(y)-p(y)+1)}{c}.
\end{cases}
\end{equation}
Since $\Phi^{-1}(-\infty)=-\infty$, we notice that $(p,q)$ satisfies the conditions
\begin{equation}\label{equilibrium}
p(-\infty)=0, \quad q(-\infty)=0.
\end{equation}
\end{description}

\begin{theorem}\label{th-uniq}
For every $c>0$ there is at most one (up to shift) semi-wavefront $(\eta,\beta)$ of~\eqref{eq-sys} defined in $(-\infty,0]$ and satisfying~\eqref{eq-bc1}.
\end{theorem}

\begin{proof}
Let $c>0$ be fixed. We deduce the uniqueness of a semi-wavefront $(\eta,\beta)$ of~\eqref{eq-sys} in $(-\infty,0]$, by using the reduction to the desingularized first-order problem~\eqref{eq-manifold} and proving that the latter system has at most one solution satisfying the boundary conditions~\eqref{equilibrium}. We thus compute the Jacobian matrix $J$ of~\eqref{eq-manifold} at the equilibrium point $(0,0)$ which is
\[
J(0,0)=\begin{bmatrix}
0		&	0\\
-c		&	-c\\
\end{bmatrix}.
\]
The eigenvalues of $J(0,0)$ are thus $\lambda_1=0$ and $\lambda_2=-c$ with associated eigenvectors $e_1=(1,-1)$ and $e_2=(0,1)$, respectively. Then by the Center Manifold Theorem (cf.,~\cite[Theorem~3.2.1]{GuHo-83}), there is a unique stable invariant manifold tangent to the eigenspace generated by $e_2$ and a not necessarily unique center invariant manifold tangent to the eigenspace generated by $e_1$.  Notice that these manifolds are one-dimensional. Since $(p,q)$ verifies~\eqref{equilibrium}, it leaves the origin through the center manifold. To prove that the center manifold is unique, we follow the approach in~\cite{AiHuang-07}. Let us assume that there are two center manifolds $W_1$ and $W_2$, then from~\cite[Theorem~10.14]{HaKo-91}, there exist two functions $k_1$ and $k_2$ of class $C^1$ such that $W_1=\left(p,k_1(p)\right)$ and $W_2=\left(p,k_2(p)\right)$ such that $k_i(0)=k'_i(0)=0$ for $i=\{1,2\}$. Since the center manifolds are invariant for the flow associated with~\eqref{eq-manifold} we obtain
\[
\dot k_i\left(p(y)\right)\frac{v\left(p(y),k_i\left(p(y)\right)\right)}{c}=-c k_i\left(p(y)\right)+\frac{v\left(p(y),k_i\left(p(y)\right)\right)}{c}, \quad i=\{1,2\},
\]
where
\[
v\left(p,k\right):=f(p,k-p+1)g(p)h(k-p+1).
\]
Let us consider
\[
\delta(y)=k_1(p(y))-k_2(p(y)),
\]
which is a bounded function in $(-\infty,0]$ because it is continuous and $\lim_{y\to-\infty}p(y)=0.$ Then, we have
\[
\dot\delta(y)+c\delta(y)=\frac{v\left(p(y),k_1(p(y))\right)-v\left(p(y),k_2(p(y))\right)}{c}.
\]
By integrating in $(-\infty,y]$ and since $\lim_{y\to-\infty}k_i\left(p(y)\right)=0$, for $i\in\{1,2\}$, we obtain
\begin{equation}\label{eq-s1}
\delta(y) e^{cy}= \int_{-\infty}^y 
\Big(v\left(p(s),k_1(p(s))\right)-v\left(p(s),k_2(p(s))\right)\Big)
\frac{e^{cs}}{c}\,\mathrm{d}s.
\end{equation}
By applying the Mean Value Theorem, we have
\begin{equation}\label{eq-s2}
v\left(p(s),k_1(p(s))\right)-v\left(p(s),k_2(p(s))\right)=
\frac{\partial v}{\partial k}\left(p(s),\bar{k}(s)	\right) \delta(s),
\end{equation}
with $\bar{k}(s)$ in between $k_1(p(s))$ and $k_2(p(s))$. We observe that 
\[
\lim_{s\to-\infty}\frac{\partial v}{\partial k}\left(p(s),\bar{k}(s)	\right)=0,
\]
because $g(p(-\infty))=0$. Thus, there is $\bar{y}\in(-\infty,0)$ such that 
\begin{equation}\label{eq-s3}
\left|\frac{\partial v}{\partial k}\left(p(s),\bar{k}(s)\right)\right|<\frac{1}{2},\quad \forall s<\bar{y}. 
\end{equation}
From~\eqref{eq-s1}, \eqref{eq-s2} and~\eqref{eq-s3}, for all $y<\bar{y}$, we have 
\[
|\delta(y)|\leq \frac{1}{2} \int_{-\infty}^y |\delta(s)|\frac{e^{c(s-y)}}{c}\,\mathrm{d}s
\leq \frac{1}{2}\sup_{y\in(-\infty,\bar{y}]} |\delta(y)|.
\]
Since $\sup_{y\in(-\infty,\bar{y}]} |\delta(y)|$ is finite, then $|\delta(y)|=0$ for all $y<\bar{y}$. This yields $k_1(p(y))=k_2(p(y))$ for all $y<\bar{y}$. By considering the initial value problems associated with
\[
\dot p=\dfrac{f(p,k_1(p)-p+1) g(p) h(k_1(p)-p+1)}{c},
\]
whose solutions are unique, we deduce that $k_1(p)=k_2(p)$ as long as they exist, namely $W_1=W_2$. This leads to the uniqueness of the solutions of \eqref{eq-manifold}--\eqref{equilibrium} of up to translation.
\end{proof}

\begin{proposition} \label{prop-conv}
Let $\{c_n\}_n$ be a sequence converging to $c_0>0$. Let $(\eta_{c_n},\beta_{c_n})$ be the semi-wavefront of \eqref{eq-sys}--\eqref{eq-bc1} in $(-\infty,0]$ associated with $c_n$ for $n\geq 0$. Then, $(\eta_{c_n},\beta_{c_n})$ converges to $(\eta_{c_0},\beta_{c_0})$ in $C(-\infty,0]$.
\end{proposition}

\begin{proof}
Since $\{c_n\}_n$ converges to $c_0>0$ there exist $m_l\leq c_n\leq m_u$, for every $n$. Without loss of generality, let us set $\eta_{c_n}(0)=1/2$,  for all $n$. Then, by Remark~\ref{rem-i}, we notice that $\beta_{c_n}(0)\geq1/2$,  for all $n$. 
Next, we divide the proof into three steps. 

Firstly, we prove that $\{\eta_{c_n}\}_n$ converge uniformly in compact sets of $(-\infty,0]$. From~\eqref{etalus}, we obtain
\begin{equation}\label{eq-nn}
\frac{1}{2} e^{\frac{L_2\xi}{m_u}}\leq \eta_{c_n}(\xi)\leq \frac{1}{2} e^{\frac{L_1(1-\eta_0) \xi}{m_l}},	\quad\forall\xi\leq 0.
\end{equation}
Moreover, from~$(A_1)$, \eqref{eq-ode-1}, and~\eqref{etalus}, we can deduce that
\[
\frac{L_1}{4 m_u} e^{\frac{L_2\xi}{m_u}}\leq \eta'_{c_n}(\xi)\leq \frac{L_2}{2m_l},	\quad\forall\xi\leq 0.
\]
Thus, $\{\eta_{c_n}\}_n$ and $\{\eta'_{c_n}\}_n$ are equibounded, and so $\{\eta_{c_n}\}_n$ is equicontinuous. By the Ascoli-Arzel\`a Theorem we conclude that, up to a subsequence $\eta_{c_n}\to\bar\eta$ as $n\to+\infty$ uniformly in compact sets of $(-\infty,0]$. 

Secondly, we prove that $\{\beta_{c_n}\}_n$ converge uniformly in compact sets of $(-\infty,0]$. From~$(A_2)$, $(A_3)$, and~\eqref{beta'}, we have
\[
-\frac{m_u e^{\frac{L_1\xi}{2m_l}}}{2 A}\leq \beta'_{c_n}(\xi)\leq 0, 		\quad\forall\xi\leq 0,
\]
where 
\[
A:=M_g g\left(\tfrac{1}{2}e^{\frac{L_2\xi}{m_u}}\right)\min_{r\in[\frac{1}{2},1]}h(r).
\]
Since, $1/2 \leq \beta_{c_n}(\xi)\leq1$, for all $\xi\leq 0$, it follows that $\{\beta_{c_n}\}_n$ and $\{\beta'_{c_n}\}_n$ are equibounded, and so $\{\beta_{c_n}\}_n$ is equicontinuous. By the Ascoli-Arzel\`a Theorem we conclude that, up to a subsequence $\beta_{c_n}\to\bar\beta$ as $n\to+\infty$ uniformly in compact sets of $(-\infty,0]$. 

At last, by showing that $(\bar\eta,\bar\beta)$ solves~\eqref{eq-sys}--\eqref{eq-bc1} for $c=c_0$, we will conclude that $(\bar\eta,\bar\beta)=(\eta_{c_0},\beta_{c_0})$. By integrating~\eqref{eq-ode-1} with $c=c_n$ in $(\xi,0)$, we have
\[
\eta_{c_n}(\xi)=\eta_{c_n}(0)-\int_\xi^0 \frac{f(\eta_{c_n}(s),\beta_{c_n}(s))}{c_n}\,\mathrm{d}s \to \bar\eta(0)-\int_\xi^0 \frac{f(\bar\eta(s),\bar\beta(s))}{c_0}\,\mathrm{d}s,
\]
and we obtain  
\[\bar\eta'(\xi)=\frac{f(\bar\eta(\xi),\bar\beta(\xi))}{c_0}, \quad\forall\xi\leq 0.\] 
By integrating~\eqref{beta'} with $c=c_n$ in $(\xi,0)$ and then passing to the limit as $n\to+\infty$ like before, we obtain 
\begin{equation}\label{eq-pom}
\bar\beta'(\xi)=\frac {c_0(1 - \bar\beta(\xi) - \bar\eta(\xi))} {g(\bar\eta(\xi)) h(\bar\beta(\xi))},  \quad\forall\xi\leq 0.
\end{equation}
From~\eqref{eq-nn}, one can prove that $\bar\eta(-\infty)=0$.  Since $\beta'_{c_n}(\xi)<0$, for every $\xi$ and for every $n$, then by passing to the limit as $n\to+\infty$, we have $\bar{\beta}'(\xi)\leq 0$, for every $\xi$. Thus, $\bar\beta$ has limit as $\xi\to-\infty$ and so we conclude, by using~\eqref{eq-pom}, $\bar{\beta}(-\infty)\geq 1$. On the other hand, $\beta_{c_n}(\xi)<1$, for every $\xi$ and for every $n$, and so by passing to the limit as $n\to+\infty$, we deduce $\bar\beta(\xi)<1.$ At last, by passing to the limit 
as $\xi\to-\infty$, we have $\bar{\beta}(-\infty)\leq 1$. We thus conclude that $\bar{\beta}(-\infty)= 1.$

Hence, we have proven that $(\bar\eta,\bar\beta)$ is a solution of~\eqref{eq-sys}--\eqref{eq-bc1} for $c=c_0$ and so, by the uniqueness of the semi-wavefronts in Theorem~\ref{th-uniq}, it follows that $(\bar\eta,\bar\beta)=(\eta_{c_0},\beta_{c_0})$, concluding the proof.
\end{proof}

%------------------------------------------------------------------------
\subsection{Extension to a classical wavefront}\label{sub-ex-p}
%------------------------------------------------------------------------

Consider the semi-wavefront $(\eta,\beta)$ of system~\eqref{eq-sys}--\eqref{eq-bc1} obtained on the negative half-line using Theorem~\ref{th-m}. It is possible to extend $(\eta,\beta)$ over the entire interval $[0, \tau)$, where $\tau \in \mathbb{R} \cup \{+\infty\}$ is defined as in~\eqref{eq-tau}. We henceforth maintain the same notation $(\eta,\beta)$ to denote the extended semi-wavefront. We will first outline several properties of the semi-wavefront. Subsequently, we will introduce a singular first-order reduction and exploit it to prove the existence of a classical wavefront satisfying~\eqref{eq-bc2}.

\begin{proposition}\label{prop-3}
Let $(\eta,\beta)$ be the extended semi-wavefront on the interval $[0, \tau)$, then it satisfies the following properties:
\begin{enumerate}[nosep,wide=0pt,  labelwidth=20pt, align=left]
\item[$(i)$]  $\eta' (\xi) > 0 $ and $\beta'(\xi) < 0 $ for every $ \xi \in [0,\tau); $
\item[$(ii)$] if $ \tau = + \infty$, then $\eta(+\infty)=1$  and $\beta(+\infty)=0;$ 
\item[$(iii)$] if $ \tau \in \mathbb R$, then $ \eta(\tau) \geq 1.$
\end{enumerate}
\end{proposition}

\begin{proof}
The assertions in $ (i)$ are a consequence of $ (A_1)$ by reasoning like in the proof of Lemma~\ref{lemma1}. 

To prove $(ii)$ and $(iii)$, we notice first that $(\eta,\beta)$ is a traveling wave of~\eqref{eq-sys} satisfying~\eqref{eq-bc1}, therefore, according to Proposition~\ref{prop1}, $\beta$ is a solution of \eqref{beta'} in $ (-\infty,\tau). $ 

To prove $(ii)$, from $(i)$ it follows the existence of $\eta(+\infty)=m \in (0,+\infty] $ and $\beta(+\infty) = l \in [0,\beta_0). $ Suppose by contradiction that $ l> 0. $ From $(A_2)$ and $(A_3)$, we obtain $h(l)g(m) \neq 0,$ thus $\beta'$ has limit as $\xi\to+\infty$, and  
\[
\beta'(+\infty) = \frac {c(1-l-m)} {g(m) h(l)}. 
\]
As a consequence, $ 1 - l - m = 0. $ In particular $ m = 1-l \in \mathbb{R}. $ We notice that $\eta'$ has limit as $\xi\to+\infty$ and  
\[
\eta'(+\infty) = \frac {f(m,l)} c. 
\]
Again it must hold $f(m,l)=0.$ Hence, from $(A_1)$, we have either $m=0$ (which is not possible because $m>0$) or $l=0$ (which is a contradiction with $l>0$). Then, it follows that $\eta(+\infty)=1$ and $\beta(+\infty)=0$. 

To prove $(iii)$, from $(i)$ and passing to the limit in~\eqref{eq-aux1} as $ \xi \to \tau$, we get that $ 1- \beta(\tau)-\eta(\tau) \leq 0,$ that is $\eta(\tau) \geq 1.$ 
\end{proof}

\begin{corollary}\label{cor-cond}
Let $(\eta,\beta)$ be the extended semi-wavefront on the interval $[0, \tau)$. Then $(\hat\eta,\hat\beta)$ defined as
\begin{equation}\label{eq-hat}
\hat{\eta}(\xi):=\begin{cases}
\eta(\xi) &\xi<\tau,\\
\eta(\tau) &\xi\geq\tau,
\end{cases}
\quad \text{and}\quad
\hat{\beta}(\xi):=\begin{cases}
\beta(\xi) &\xi<\tau,\\
0 &\xi\geq\tau,
\end{cases}\end{equation} 
is a wavefront of~\eqref{eq-sys}--\eqref{eq-bc} if and only if 
\begin{equation}\label{eq-lim}
\lim_{\xi\to\tau^-} h(\hat\beta(\xi))\hat\beta'(\xi)=0,
\end{equation}
or equivalently,
\begin{equation}\label{eq-st}
\tau=\sigma:=\inf\{\xi\in\mathbb{R}\colon \eta(\xi)=1\}.
\end{equation}
\end{corollary}

\begin{proof}
Firstly, we prove the if condition~\eqref{eq-lim} holds, then $(\hat\eta,\hat\beta)$ verifies Definition~\ref{def-wave} and~\eqref{eq-bc2}. If $\tau=+\infty$, according with Proposition~\ref{prop1}, it is sufficient to show that~\eqref{eq-bc2} holds. By Proposition~\ref{prop-3}, we have $\hat\beta(+\infty)=0$ and $\hat\eta(+\infty)=1$. On the other hand, if $\tau<+\infty$, by passing to the limit in~\eqref{eq-aux1} as $ \xi \to \tau$ and using~\eqref{eq-lim}, we get $0= 1- \hat{\beta}(\tau)-\hat{\eta}(\tau) = 1- \hat{\eta}(\tau),$ and so $\hat\eta(\tau)=1$. It remains to prove that $\hat\beta$ verifies~\eqref{eq-weak}. Since $\hat\beta$ is of class $C^2$ in $(-\infty,\tau)$ by Proposition~\ref{prop1}, without loss of generality, let $\varepsilon>0$ and $\psi\in C^{\infty}_0(\tau-\varepsilon,\tau+\varepsilon)$, then
\[\begin{split}
&\int^{\tau+\varepsilon}_{\tau-\varepsilon} \Big(\left(g(\hat\eta(\xi))h(\hat\beta(\xi))\hat\beta'(\xi)+c\hat\beta(\xi)\right)\psi'(\xi) -f(\hat\eta(\xi),\hat\beta(\xi))\psi(\xi) \Big)\,\mathrm{d}\xi	\\
&=\lim_{\delta\to0} \int^{\tau-\delta}_{\tau-\varepsilon} \Big(\left(g(\hat\eta(\xi))h(\hat\beta(\xi))\hat\beta'(\xi)+c\hat\beta(\xi)\right)\psi'(\xi) -f(\hat\eta(\xi),\hat\beta(\xi))\psi(\xi) \Big)\,\mathrm{d}\xi	\\
&=\lim_{\delta\to0} \left[\left(g(\hat\eta(\xi))h(\hat\beta(\xi))\hat\beta'(\xi)+c\hat\beta(\xi)\right)\psi(\xi)\right]^{\tau-\delta}_{\tau-\varepsilon}  \\
	&\quad-\lim_{\delta\to0}\int^{\tau-\delta}_{\tau-\varepsilon} \Big(\left(g(\hat\eta(\xi))h(\hat\beta(\xi))\hat\beta'(\xi)\right)'+c\hat\beta(\xi) -f(\hat\eta(\xi),\hat\beta(\xi))\Big) \psi(\xi)\,\mathrm{d}\xi\\
&= \lim_{\delta\to0} \Big(\left(g(\hat\eta)h(\hat\beta)\hat\beta'\right)(\tau-\delta) +c\hat\beta(\tau-\delta)\Big)\psi(\tau-\delta)\\
&= g(1) \psi(\tau) \lim_{\delta\to0} \left(h(\hat\beta)\hat\beta'\right)(\tau-\delta)=0.
\end{split}\]
This concludes that $(\hat\eta,\hat\beta)$ is a wavefront of~\eqref{eq-sys}--\eqref{eq-bc}. 

On the other hand, we observe that if $(\hat\eta,\hat\beta)$ is a wavefront of~\eqref{eq-sys}--\eqref{eq-bc}, then condition~\eqref{eq-lim} follows from Lemma~\ref{lemma1bis}. 

At last, the equivalence between conditions~\eqref{eq-lim} and~\eqref{eq-st} holds by using~\eqref{eq-aux1} and Proposition~\ref{prop-3}. Indeed, \eqref{eq-lim} is verified if and only if $\lim_{\xi\to\tau^-}c(1-\beta(\xi)-\eta(\xi))=0$, that is $\lim_{\xi\to\tau^-}\eta(\xi)=1$, namely $\tau=\sigma$.
\end{proof}

\begin{remark} \label{rem-figures}
Proposition~\ref{prop-3} states that if $\tau = +\infty$, then $(\hat{\eta},\hat{\beta})=({\eta},{\beta})$ is indeed a classical wavefront of system~\eqref{eq-sys}--\eqref{eq-bc} on $\mathbb{R}$. However, when $\tau \in \mathbb{R}$, the status of $(\hat{\eta},\hat{\beta})$ as a wavefront is uncertain. In such cases, it is possible that $(\hat{\eta},\hat{\beta})$ may or may not be a wavefront. In the latter scenario, we have $\tau \neq \sigma$, where ${\eta}(\tau) > 1$, $\eta(\sigma) = 1$, and $\beta(\sigma) > 0$. It is important to note that since the finiteness of $\tau$ is not known a priori, condition~\eqref{eq-lim} or~\eqref{eq-st} is employed to ensure that the extended semi-wavefront is a wavefront. 
In Figure~\ref{fig-extension}, we depict the various possible configurations. We specifically avoid the case where $\beta'(\tau)=0$ for $\tau\in\mathbb{R}$ since, in Section~\ref{section-4}, we will demonstrate that such a case cannot hold. Accordingly, the wavefront is sharp if $\tau\in\mathbb{R}$.
\end{remark} 

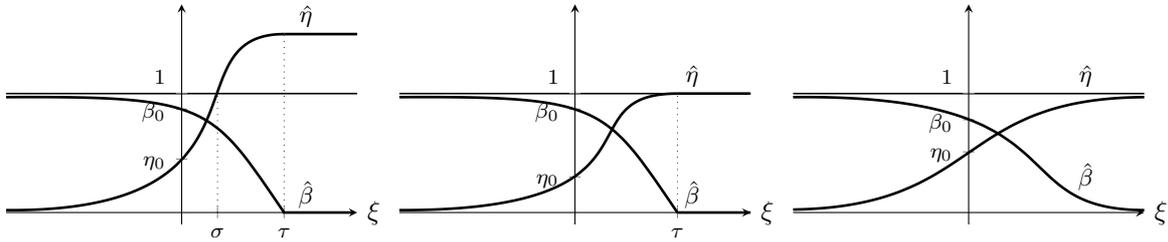
\begin{figure}[h!]
\begin{tikzpicture}
\begin{axis}[legend style={at={(axis cs:72,30)},anchor=north west},
  tick label style={font=\scriptsize},
  axis y line=middle, 
  axis x line=middle,
  ytick={0,1},
  yticklabel style={anchor=south east},
  xtick={0,2.45,7},
  xticklabels={$0$,$\sigma$,$\tau$}, 
  extra y tick style={anchor=north east,yticklabel style={anchor=east,yshift=-0.5mm}}, 
  extra y ticks={0.45,0.87,1},
  extra y tick labels={$\eta_0$,$\beta_0$},
  xlabel={\small $\xi$}, ylabel={},
every axis x label/.style={
    at={(ticklabel* cs:1.0)},
    anchor=west,
},
every axis y label/.style={
    at={(ticklabel* cs:5.0)},
    anchor=south west
},
  set layers,
  width=6.2cm,
  height=4.5cm,
  xmin=-12,
  xmax=12,
  ymin=-0.1,
  ymax=1.75]
\addplot [draw=black, line width=0.6pt, smooth, on layer=axis background]coordinates {(-12,1)(12,1)}; 
\addplot [draw=black, dotted, line width=0.3pt, smooth, on layer=axis background]coordinates {(7,0.000)(7,1.5)}; 
\addplot [draw=black, dotted, line width=0.3pt, smooth, on layer=axis background]coordinates {(2.45,0.000)(2.45,1)}; 
\addplot [draw=black, line width=1pt, smooth]coordinates {(6.9,0.000) (11.8,0)};
\addplot [draw=black, line width=1pt, smooth]coordinates {(-12.0000,0.9700) (-11.3540,0.9704) (-10.7280,0.9707) (-10.1220,0.9707) (-9.5344,0.9706) (-8.9660,0.9703) (-8.4160,0.9698) (-7.8838,0.9691) (-7.3691,0.9681) (-6.8715,0.9668) (-6.3904,0.9653) (-5.9255,0.9635) (-5.4763,0.9613) (-5.0424,0.9588) (-4.6234,0.9560) (-4.2188,0.9528) (-3.8281,0.9492) (-3.4510,0.9452) (-3.0871,0.9408) (-2.7358,0.9360) (-2.3967,0.9307) (-2.0695,0.9250) (-1.7536,0.9188) (-1.4487,0.9120) (-1.1543,0.9048) (-0.8700,0.8970) (-0.5952,0.8887) (-0.3297,0.8798) (-0.0730,0.8703) (0.1754,0.8602) (0.4159,0.8495) (0.6490,0.8382) (0.8750,0.8263) (1.0944,0.8136) (1.3077,0.8003) (1.5153,0.7863) (1.7175,0.7715) (1.9149,0.7560) (2.1079,0.7398) (2.4824,0.7051) (2.6648,0.6865) (2.8445,0.6671) (3.1975,0.6258) (3.3718,0.6039) (3.5451,0.5810) (3.8906,0.5327) (4.0637,0.5071) (4.2376,0.4806) (4.5896,0.4246) (5.3223,0.3005) (5.5140,0.2668) (5.7099,0.2321) (6.1165,0.1593) (6.3279,0.1212) (6.5453,0.0820) (6.7692,0.0416) (7.0000,0.0000)};
\addplot [draw=black, line width=1pt, smooth]coordinates {(12,1.5) (7,1.5)};
\addplot [draw=black, line width=1pt, smooth]coordinates {(7.0000,1.5000) (6.6366,1.4989) (6.2956,1.4957) (5.9761,1.4904) (5.6771,1.4832) (5.3976,1.4740) (5.1367,1.4630) (4.8933,1.4502) (4.6664,1.4356) (4.4551,1.4195) (4.2585,1.4017) (4.0754,1.3825) (3.9049,1.3618) (3.7461,1.3397) (3.5979,1.3164) (3.4594,1.2918) (3.3295,1.2661) (3.2074,1.2393) (3.0919,1.2114) (2.8773,1.1529) (2.7760,1.1224) (2.6776,1.0912) (2.4851,1.0267) (2.3890,0.9936) (2.2918,0.9601) (2.0899,0.8919) (1.9833,0.8574) (1.8715,0.8226) (1.6288,0.7529) (1.4958,0.7180) (1.3537,0.6832) (1.2017,0.6485) (1.0386,0.6141) (0.8635,0.5799) (0.6754,0.5462) (0.2563,0.4799) (0.0234,0.4476) (-0.2265,0.4160) (-0.4943,0.3850) (-0.7810,0.3548) (-1.0875,0.3255) (-1.4150,0.2971) (-1.7643,0.2696) (-2.1364,0.2432) (-2.5324,0.2180) (-2.9531,0.1939) (-3.3997,0.1710) (-3.8730,0.1495) (-4.3741,0.1294) (-4.9039,0.1108) (-5.4635,0.0937) (-6.0538,0.0781) (-6.6758,0.0643) (-7.3304,0.0522) (-8.0188,0.0419) (-8.7418,0.0335) (-9.5004,0.0270) (-10.2960,0.0226) (-11.1290,0.0202) (-12.0000,0.0200)};
\node at (axis cs:8.5,0.15) {\footnotesize{$\hat\beta$}};
\node at (axis cs:8.5,1.65) {\footnotesize{$\hat\eta$}};
\end{axis}
\end{tikzpicture}
\begin{tikzpicture}
\begin{axis}[legend style={at={(axis cs:72,30)},anchor=north west},
  tick label style={font=\scriptsize},
  axis y line=middle, 
  axis x line=middle,
  ytick={0,1},
  yticklabel style={anchor=south east},
  xtick={0,7},
  xticklabels={$0$,$\tau$}, 
  extra y tick style={anchor=north east,yticklabel style={anchor=east,yshift=-0.5mm}}, 
  extra y ticks={0.3,0.87,1},
  extra y tick labels={$\eta_0$,$\beta_0$},
  xlabel={\small $\xi$}, ylabel={},
every axis x label/.style={
    at={(ticklabel* cs:1.0)},
    anchor=west,
},
every axis y label/.style={
    at={(ticklabel* cs:5.0)},
    anchor=south west
},
  set layers,
  width=6.2cm,
  height=4.5cm,
  xmin=-12,
  xmax=12,
  ymin=-0.1,
  ymax=1.75]
\addplot [draw=black, line width=0.6pt, smooth, on layer=axis background]coordinates {(-12,1)(12,1)}; 
\addplot [draw=black, dotted, line width=0.3pt, smooth, on layer=axis background]coordinates {(7,0)(7,1)}; 
\addplot [draw=black, line width=1pt, smooth]coordinates {(12,0) (7,0)};
\addplot [draw=black, line width=1pt, smooth]coordinates {(-12.0000,0.9700) (-11.3540,0.9704) (-10.7280,0.9707) (-10.1220,0.9707) (-9.5344,0.9706) (-8.9660,0.9703) (-8.4160,0.9698) (-7.8838,0.9691) (-7.3691,0.9681) (-6.8715,0.9668) (-6.3904,0.9653) (-5.9255,0.9635) (-5.4763,0.9613) (-5.0424,0.9588) (-4.6234,0.9560) (-4.2188,0.9528) (-3.8281,0.9492) (-3.4510,0.9452) (-3.0871,0.9408) (-2.7358,0.9360) (-2.3967,0.9307) (-2.0695,0.9250) (-1.7536,0.9188) (-1.4487,0.9120) (-1.1543,0.9048) (-0.8700,0.8970) (-0.5952,0.8887) (-0.3297,0.8798) (-0.0730,0.8703) (0.1754,0.8602) (0.4159,0.8495) (0.6490,0.8382) (0.8750,0.8263) (1.0944,0.8136) (1.3077,0.8003) (1.5153,0.7863) (1.7175,0.7715) (1.9149,0.7560) (2.1079,0.7398) (2.4824,0.7051) (2.6648,0.6865) (2.8445,0.6671) (3.1975,0.6258) (3.3718,0.6039) (3.5451,0.5810) (3.8906,0.5327) (4.0637,0.5071) (4.2376,0.4806) (4.5896,0.4246) (5.3223,0.3005) (5.5140,0.2668) (5.7099,0.2321) (6.1165,0.1593) (6.3279,0.1212) (6.5453,0.0820) (6.7692,0.0416) (7.0000,0.0000)};
\addplot [draw=black, line width=1pt, smooth]coordinates {(12,1) (7,1)};
\addplot [draw=black, line width=1pt, smooth]coordinates {(7.0000,1.0000) (6.6366,0.9993) (6.2956,0.9972) (5.9761,0.9937) (5.6771,0.9889) (5.3976,0.9829) (5.1367,0.9757) (4.8933,0.9672) (4.6664,0.9577) (4.4551,0.9471) (4.2585,0.9354) (4.0754,0.9227) (3.9049,0.9091) (3.7461,0.8946) (3.5979,0.8793) (3.4594,0.8631) (3.3295,0.8462) (3.2074,0.8285) (3.0919,0.8102) (2.8773,0.7717) (2.7760,0.7516) (2.6776,0.7310) (2.4851,0.6886) (2.3890,0.6668) (2.2918,0.6447) (2.0899,0.5998) (1.9833,0.5770) (1.8715,0.5541) (1.6288,0.5081) (1.4958,0.4851) (1.3537,0.4621) (1.2017,0.4393) (1.0386,0.4165) (0.8635,0.3940) (0.6754,0.3717) (0.2563,0.3279) (0.0234,0.3065) (-0.2265,0.2856) (-0.4943,0.2651) (-0.7810,0.2451) (-1.0875,0.2256) (-1.4150,0.2068) (-1.7643,0.1886) (-2.1364,0.1710) (-2.5324,0.1542) (-2.9531,0.1382) (-3.3997,0.1229) (-3.8730,0.1086) (-4.3741,0.0951) (-4.9039,0.0826) (-5.4635,0.0711) (-6.0538,0.0607) (-6.6758,0.0513) (-7.3304,0.0431) (-8.0188,0.0360) (-8.7418,0.0302) (-9.5004,0.0256) (-10.2960,0.0224) (-11.1290,0.0205) (-12.0000,0.0200)};
\node at (axis cs:8,0.15) {\footnotesize{$\hat\beta$}};
\node at (axis cs:8,1.15) {\footnotesize{$\hat\eta$}};
\end{axis}
\end{tikzpicture}
\begin{tikzpicture}
\begin{axis}[legend style={at={(axis cs:72,30)},anchor=north west},
  tick label style={font=\scriptsize},
  axis y line=middle, 
  axis x line=middle,
  ytick={0,1},
  yticklabel style={anchor=south east},
  xtick={0,12},
  xtick style={draw=none},
  xticklabels={0,\color{white}$\tau$}, 
  extra y tick style={anchor=north east,yticklabel style={anchor=east,yshift=-0.5mm}}, 
  extra y ticks={0.51,0.78,1},
  extra y tick labels={$\eta_0$,$\beta_0$},
  xlabel={\small $\xi$}, ylabel={},
every axis x label/.style={
    at={(ticklabel* cs:1.0)},
    anchor=west,
},
every axis y label/.style={
    at={(ticklabel* cs:5.0)},
    anchor=south west
},
  set layers,
  width=6.2cm,
  height=4.5cm,
  xmin=-12,
  xmax=12,
  ymin=-0.1,
  ymax=1.75]
\addplot [draw=black, line width=0.6pt, smooth, on layer=axis background]coordinates {(-12,1)(12,1)}; 
\addplot [draw=black, line width=1pt, smooth]coordinates {(-12.0000,0.9700) (-11.0810,0.9705) (-10.2000,0.9697) (-9.3546,0.9674) (-8.5444,0.9638) (-7.7685,0.9590) (-7.0259,0.9529) (-6.3157,0.9456) (-5.6367,0.9371) (-4.9882,0.9276) (-4.3691,0.9170) (-3.7785,0.9054) (-3.2153,0.8928) (-2.6788,0.8793) (-2.1678,0.8649) (-1.6814,0.8497) (-1.2188,0.8337) (-0.7788,0.8169) (-0.3605,0.7995) (0.0369,0.7813) (0.4146,0.7626) (0.7733,0.7434) (1.1142,0.7236) (1.7461,0.6826) (2.0390,0.6615) (2.3179,0.6401) (2.8374,0.5964) (3.0799,0.5743) (3.3122,0.5519) (3.7500,0.5070) (3.9574,0.4845) (4.1585,0.4620) (4.5454,0.4172) (5.2852,0.3299) (5.4686,0.3088) (5.6534,0.2881) (6.0308,0.2480) (6.2253,0.2287) (6.4250,0.2100) (6.8438,0.1743) (7.0647,0.1576) (7.2947,0.1415) (7.5347,0.1263) (7.7856,0.1118) (8.0485,0.0983) (8.3242,0.0857) (8.6137,0.0740) (8.9180,0.0634) (9.2380,0.0538) (9.5748,0.0453) (9.9292,0.0379) (10.3020,0.0318) (10.6950,0.0269) (11.1080,0.0232) (11.5430,0.0209) (12.0000,0.0200)};
\addplot [draw=black, line width=1pt, smooth]coordinates {(12.0000,0.9700) (11.4010,0.9689) (10.8240,0.9665) (10.2680,0.9628) (9.7310,0.9580) (9.2135,0.9520) (8.7145,0.9449) (8.2335,0.9368) (7.7695,0.9275) (7.3221,0.9174) (6.8904,0.9062) (6.4739,0.8941) (6.0718,0.8812) (5.6834,0.8675) (5.3080,0.8529) (4.5938,0.8216) (4.2534,0.8050) (3.9234,0.7877) (3.2915,0.7514) (2.9883,0.7325) (2.6926,0.7131) (2.1211,0.6732) (1.8439,0.6527) (1.5716,0.6319) (1.0386,0.5896) (0.0000,0.5031) (-0.2579,0.4814) (-0.5165,0.4597) (-1.0386,0.4165) (-1.3034,0.3951) (-1.5716,0.3740) (-2.1211,0.3324) (-2.4037,0.3120) (-2.6926,0.2920) (-3.2915,0.2532) (-3.6030,0.2346) (-3.9234,0.2164) (-4.5938,0.1818) (-4.9451,0.1654) (-5.3080,0.1498) (-5.6834,0.1348) (-6.0718,0.1206) (-6.4739,0.1073) (-6.8904,0.0947) (-7.3221,0.0831) (-7.7695,0.0724) (-8.2335,0.0626) (-8.7145,0.0539) (-9.2135,0.0462) (-9.7310,0.0396) (-10.2680,0.0342) (-10.8240,0.0299) (-11.4010,0.0268) (-12.0000,0.0250)};
\node at (axis cs:8,0.3) {\footnotesize{$\hat\beta$}};
\node at (axis cs:8,1.15) {\footnotesize{$\hat\eta$}};
\end{axis}
\end{tikzpicture}
\caption{Possible configurations of the extended solution $(\hat{\eta},\hat{\beta})$ as defined in~\eqref{eq-hat}: exploring the behavior at $\tau$ to determine if $(\hat{\eta},\hat{\beta})$ satisfies~\eqref{eq-sys}--\eqref{eq-bc} on $\mathbb{R}$, resulting in non-solution (left), sharp (center), or classical (right) categories. }\label{fig-extension}
\end{figure}

To verify whether $(\hat{\eta},\hat{\beta})$ is a classical wavefront of system~\eqref{eq-sys}--\eqref{eq-bc} on $\mathbb{R}$, we rely on an established result reported below due to~\cite{ArWe-78}. 

\begin{theorem}\label{th-fkpp}
Let $\rho\in\mathbb{R}$, $K>1$, and $\gamma\in C^1([0, K])$ be such that $\gamma(0)=\gamma(K)=0$, $\gamma(s) > 0$ for all $s\in(0,K)$, and $\dot\gamma(0) > 0$. Then, 
\begin{equation}\label{eq-w}
\begin{cases}
\dot w(\beta)=-\rho-\dfrac{\gamma(\beta)}{w(\beta)}, & 0<\beta<K,\\
w(\beta)<0, & 0<\beta<K,\\
w(0)=w(K)=0,
\end{cases}
\end{equation}
is solvable if and only if $\rho\geq \rho_*$, with
\begin{equation}\label{eq-rho}
2\sqrt{\dot\gamma(0)}\leq \rho_*\leq 2\sqrt{\sup_{s\in(0,K]}\frac{\gamma(s)}{s}}.
\end{equation}
Moreover, the solution $w$ is unique and satisfies 
\begin{equation}\label{eq-derw}
\dot w(0)=\frac{-\rho-\sqrt{\rho^2-4\gamma'(0)}}{2}.
\end{equation}
\end{theorem}

The proof of this theorem can be found in references such as~\cite{ArWe-78, Fi-79, MaMa-03}. By applying this result, we exploit a singular first-order reduction technique.

\begin{description}[leftmargin=*]
\item[Reduction to a singular first-order problem.]
For every $c>0$, let $(\eta,\beta)$ be the extended semi-wavefront on the interval $[0, \tau)$. By Proposition~\ref{prop-3}, it follows that $\beta$ is a strict monotone function in $[0, \tau)$. Hence, we can define
\[N(\beta):=\eta\left(\xi(\beta)\right), \quad 0< \beta\leq \beta_0,
\]
where $\xi=\xi(\beta)$ is the inverse function of $\beta$. Let 
\begin{equation}\label{eq-z}
z_c(\beta):=g\left(N(\beta)\right)h(\beta)\beta'\left(\xi(\beta)\right), \quad 0<\beta\leq\beta_0.
\end{equation}
Notice that $z_c(\beta)<0$ for all $\beta\in(0,\beta_0)$. From~\eqref{eq-ode-2}, we obtain
\begin{equation}\label{eq-z}
\dot z_c(\beta)=-c-\frac{f(N(\beta), \beta) g(N(\beta)) h(\beta)}{z_c(\beta)}, \quad 0<\beta\leq\beta_0.
\end{equation}
\end{description}

Let $K>1$ be fixed. In the following result, we will exploit Theorem~\ref{th-fkpp} and so we define $\gamma\colon [0, K] \to \mathbb{R}$ such that
\[
\gamma(\beta):=\begin{cases}
\displaystyle L_2 h(\beta) \max_{[0,1]}g,& \beta\in[0,1],\\
\gamma_0(\beta),& \beta\in(1,K],
\end{cases}\]
where $\gamma_0(\beta)$ is any $C^1$ strictly positive function on $(1,K)$ that smoothly connects the points $(1,\gamma(1))$ and $(K,0)$. This way $\gamma\in C^1([0,K])$ with $\gamma(0)=\gamma(K)=0$, $\gamma(\beta)>0$ for every $\beta\in(0,K)$, and $\dot\gamma(0)>0$ by $(A_3)$. 

\begin{theorem}\label{th-positive}
There is $\rho_*>0$ such that, for all $c>\rho_*$, $(\hat\eta,\hat\beta)$ is a classical wavefront of~\eqref{eq-sys}--\eqref{eq-bc}.
\end{theorem}

\begin{proof}
By Theorem~\ref{th-fkpp}, problem~\eqref{eq-w} has a unique solution $w_{\rho_*}$ for $\rho=\rho_*$ satisfying~\eqref{eq-rho}. Let $c>\rho_*$ be fixed. Given $\delta\in(0,1)$, let us take
\[
\mu_\delta:=-\max_{\beta\in[\delta,1]}w_{\rho_*}(\beta).
\] 
Since $w_{\rho_*}(0)=0$, there is $\delta_0\in(0,1)$ such that $\mu_{\delta_0}<\rho_*.$ Let us take
\[
\eta_0:=(1-\delta_0)\frac{\mu_{\delta_0}}{c}.
\]
By using Theorem~\ref{th-m}, let $(\eta,\beta)$ be the extended semi-wavefront on $(-\infty,\tau)$ passing through $(\eta_0,\beta_0)$. Notice that $\eta_0<1-\delta_0$ and so $\beta_0\geq 1-\eta_0>\delta_0$. 

Let us consider the solution $z_c$ of~\eqref{eq-z} obtained by reducing to the singular first-order problem, as reported above. Next, we divide the proof into three steps.

\noindent\textit{Step~$1$.} We claim that $z_c(0)=0$ and 
\begin{equation}\label{eq-zw}
z_c(\beta)>w_{\rho_*}(\beta),\quad \forall\beta\in(0,\beta_0).
\end{equation}
To this purpose, from~\eqref{eq-aux1}, notice that
\[
z_c(\beta_0)=c-c\beta_0-cN(\beta_0)>-cN(\beta_0)=-(1-\delta_0)\mu_{\delta_0}>-\mu_{\delta_0}\geq w_{\rho_*}(\beta_0).
\]
To prove the claim, we argue by contradiction, and so we suppose that there is $\beta_1\in(0,\beta_0)$ such that 
\begin{equation}\label{eq-dd}
z_c(\beta_1)-w_{\rho_*}(\beta_1)=0, \quad z_c(\beta)-w_{\rho_*}(\beta)>0, \quad  \forall\beta\in(\beta_1,\beta_0).
\end{equation} 
Since 
\[\gamma(\beta_1)=L_2 h(\beta_1) \max_{[0,1]}g\geq f\left(N(\beta_1),\beta_1\right)g\left(N(\beta_1)\right)h(\beta_1),\]
and $w_{\rho_*}(\beta_1)<0$, then we deduce that
\[\begin{split}
\dot z_c(\beta_1)-\dot  w_{\rho_*}(\beta_1)
&=-c-\frac{f(N(\beta_1), \beta_1) g(N(\beta_1)) h(\beta_1)}{z(\beta_1)}+\rho_*+\frac{\gamma(\beta_1)}{w_{\rho_*}(\beta_1)}\\
&=\rho_*-c+\frac{\gamma(\beta_1)-f(N(\beta_1), \beta_1) g(N(\beta_1)) h(\beta_1)}{w_{\rho_*}(\beta_1)}<0,
\end{split}\]
which is in contradiction with~\eqref{eq-dd} and this proves~\eqref{eq-zw}. An immediate consequence is that 
\[0= \lim_{\beta\to0^+}w_{\rho_*}(\beta)\leq \lim_{\beta\to0^+}z_c(\beta)=c(1-\eta(\tau))\leq0,
\] 
by Proposition~\ref{prop1} and Proposition~\ref{prop-3}. Hence $z_c(0)=0$, concluding the proof of the claim. 

\noindent\textit{Step~$2$.} We claim that the couple of functions $(\hat{\eta},\hat{\beta})$ defined as in~\eqref{eq-hat} is a wavefront of system~\eqref{eq-sys}--\eqref{eq-bc}.
Thanks to~\eqref{eq-z} and the previous step, we obtain
\[\begin{split}
0&=\lim_{\beta\to0^+}z_c(\beta)=\lim_{\beta\to0^+} g(N(\beta)) h(\beta)\beta'\left(\xi(\beta)\right)\\
&=g(1)\lim_{\xi\to\tau^-} h(\beta(\xi)) \beta'(\xi)= g(1)\lim_{\xi\to\tau^-} h(\hat\beta(\xi))\hat\beta'(\xi).
\end{split}\]
By $(A_2)$, since $g(1)>0$, we have $\lim_{\xi\to\tau^-} h(\hat\beta(\xi))\hat\beta'(\xi)=0$. Hence, based on Corollary~\ref{cor-cond}, we proved the claim. 

\noindent\textit{Step~$3$.} We claim that $(\hat\eta,\hat\beta)$ is a classical wavefront. Notice that 
\[
\dot z_c(0)=\lim_{\beta\to0^+} \frac{z_c(\beta)}{\beta}
=g(1)\dot h(0)\lim_{\xi\to\tau^-} \beta'(\xi).
\]
By \eqref{eq-alt}, $\lim_{\beta\to0^+}\dot z_c(\beta)$ equals $0$ or $-c$. From \textit{Step~$1$} and~\eqref{eq-derw}, $\dot z_c(0)\geq \dot w_{\rho^*}(0)\geq -\rho^*>-c$. This way, $\dot z_c(0)=0$ and, from $(A_2)$--$(A_3)$, it follows $\beta'(\tau)=0$. Therefore, $(\hat\eta,\hat\beta)$ is a classical wavefront and this ends the proof.
\end{proof}

\begin{remark} \label{rem-dotz}
According to Corollary \ref{cor-cond} and the definition of $z_c$ in~\eqref{eq-z}, $(\eta,\beta)$ is a wavefront of \eqref{eq-sys}--\eqref{eq-bc}  if and only if $ z_c(0)=0. $ Moreover, in this case, we have
\[ g(1)\dot h(0)\lim_{\xi\to\tau^-} \beta'(\xi) =\dot z_c(0) = c\lim_{\xi \to \tau^-} \biggl[ -1 - \frac {\eta'(\xi)} {\beta'(\xi)} \biggr].
\] 
As a consequence, by recalling Remark~\ref{rem-sharp}, we deduce that the wavefront is classical if and only if
\[ \lim_{\xi \to \tau^-} \frac {\eta'(\xi)} {\beta'(\xi)}= -1, \]
and the wavefront is sharp if and only if
\[ \lim_{\xi \to \tau^-} \frac {\eta'(\xi)} {\beta'(\xi)}= 0. \]
\end{remark}

\begin{remark}\label{rem-star}
By the arbitrariness of the choice of $\gamma_0$, the existence of a classical wavefront of~\eqref{eq-sys}--\eqref{eq-bc} holds for all $c>c_*$, where
\begin{equation}\label{c-star}
c_*:=2\sqrt{L_2\max_{s\in[0,1]}g(s)\sup_{r\in(0,1]}\frac{h(r)}{r}}.
\end{equation}
In fact, we can choose a sequence $\{\gamma_{0,n}\}_n$ where $\gamma_{0,n} \colon[1,K_n]\to\mathbb{R}$ is of class $C^1$ and such that $\gamma_{0,n}(1)=L_2(\max_{[0,1]}g)h(1)$, $\dot\gamma_{0,n}(1)=L_2(\max_{[0,1]}g)\dot h(1)$, $\gamma_{0,n}(K_n)=0$, and 
\[
\lim_{n\to+\infty}\max_{s\in[1,K_n]} \gamma_{0,n}(s)= L_2\max_{s\in[0,1]}g(s)h(1).
\]
Then, by taking
\[\gamma_n(\beta):=\begin{cases}
\displaystyle L_2  (\max\nolimits_{[0,1]}g) h(\beta),& \beta\in[0,1],\\
\gamma_{0,n}(\beta),& \beta\in(1,K_n],
\end{cases}\]
it follows that 
\[\lim_{n\to+\infty}\sqrt{\sup_{s\in(0,K_n]}\frac{\gamma_n(s)}{s}}= \sqrt{L_2\max_{s\in[0,1]}g(s)\sup_{r\in(0,1]}\frac{h(r)}{r}}.
\]
Consequently, by applying Theorem~\ref{th-positive} for every $n$, we can deduce there is a classical wavefront of~\eqref{eq-sys}--\eqref{eq-bc} for all $c>\rho_{*,n}$, where 
\[\rho_{*,n}\leq 2\sqrt{\sup_{s\in(0,K_n]}\frac{\gamma_n(s)}{s}}.\]
Then, by passing to the limit as $n\to+\infty$, we obtain~\eqref{c-star}.
\end{remark}

%-_-_-_-_-_-_-_-_-_-_-_-_-_-_-_-_-_-_-_-_-_-_-_-_-_-_-_-_-_-_-_-_-_-_
\section{Qualitative properties of wavefronts}\label{section-4}
%-_-_-_-_-_-_-_-_-_-_-_-_-_-_-_-_-_-_-_-_-_-_-_-_-_-_-_-_-_-_-_-_-_-_

Regarding assumptions $(A_1)$--$(A_3)$, one can assume the regularity of the functions in a neighborhood of $0$ and $1$.

Let us consider the set of admissible speeds:
\[
\mathcal{C}=\{c>0\colon \text{\eqref{eq-sys}--\eqref{eq-bc} has a wavefront $(\eta,\beta)$ in }\mathbb{R}\}.
\] 

\begin{remark}\label{rem41}
According to Theorem~\ref{th-positive} and Remark~\ref{rem-star}, we observe that the set $\mathcal{C}$ is non-empty, specifically containing the interval $(c_*,+\infty)$, where $c_*$ is defined as in~\eqref{c-star}. Furthermore, we can show that $c_\sharp<c_*$, where $c_\sharp$ is defined as in~\eqref{est-c}. Accordingly, 
\[
\int_0^1 g(1-r)h(r)r\,\mathrm{d}r\leq \frac{1}{3}\max_{[0,1]}g \sup_{(0,1]}\frac{h(s)}{s}, \quad \int_0^1 (1-r)g(1-r)h(r)r\,\mathrm{d}r\leq \frac{1}{20}\max_{[0,1]}g \sup_{(0,1]}\frac{h(s)}{s}.
\]
In this manner, since $M_g\leq1$ and $L_1\leq L_2$, we obtain that 
\[
c_{\sharp}\leq \sqrt{L_1 M_g\frac{1}{3}\max_{[0,1]}g \sup_{(0,1]}\frac{h(s)}{s}} \leq 2 \sqrt{L_2\max_{[0,1]}g \sup_{(0,1]}\frac{h(s)}{s}}=c_*.
\]
As a consequence, we have
\[
(-\infty,c_\sharp)\cap\mathcal{C}=\emptyset, \qquad \left( c_*, + \infty \right) \subset \mathcal C, \qquad \left[c_\sharp, c_*\right]\not=\emptyset.
\]
In the general case, it is not possible to guarantee that $ c_\sharp \leq 2 \sqrt{\dot{\gamma}(0)}. $ However, when $ f(s,r)=sr, g(s)=s $, and $ h(r)=r, $ like in \cite{SMGA-01}, $ \gamma(s) = s,$ for every $s\in(0,1]$ and, according to Remark~\ref{rem-csharp}, $ c_\sharp = \sqrt {\frac 1 {12}} < 2 = 2 \sqrt{\dot {\gamma}(0)}=c_*.$ We stress that in~\cite{SMGA-01}, only a lower bound of the threshold speed is provided while no upper bound is given. 
\end{remark}

In this section, we extensively use the following reduction to prove both the connection and the closure of $\mathcal{C}$.

\begin{description}[leftmargin=*]
\item[Reduction to another singular first-order problem.]
For every $c>0$, let $(\eta,\beta)$ be the extended semi-wavefront on the interval $[0, \sigma)$, where $\sigma$ is defined as in~\eqref{eq-st}. Instead of looking at $\beta$, as in Section~\ref{sub-ex-p}, now again using Proposition~\ref{prop-3}, we observe that $\eta$ is a strict monotone function in $[0, \sigma)$. Hence, we can define
\[B_c(\eta):=\beta\left(\xi(\eta)\right), \quad 0< \eta< 1,
\]
where $\xi=\xi(\eta)$ is the inverse function of $\eta$. Notice that $0< B_c(\eta)<1$, for all $\eta\in(0,1)$. From~\eqref{eq-aux1} and~\eqref{eq-bc1}, we obtain
\begin{equation}\label{eq-bb}
\begin{cases}
\dot B_c(\eta)=\dfrac{c^2\left(1-\eta-B_c(\eta)\right)}{g(\eta)h(B_c(\eta)) f(\eta,B_c(\eta))}, \quad 0 < \eta <1, \\
B_c(0)=1.
\end{cases}
\end{equation}
\end{description}

\begin{lemma} \label{lem-dotB}
For all $c>0$, $\dot B_c(0)=-1.$	
\end{lemma}

\begin{proof}
Assume by contradiction that $\lim_{\eta \to 0} \dot B_c(\eta)$ does not exist, then we have also that the limit of the following $C^1$ function
\[
\psi(\eta)=\frac{1-B_c(\eta)-\eta}{g(\eta) f(\eta,B_c(\eta))}
\]
does not exist, since
\begin{equation}\label{eq-psi1}
\dot B_c(\eta)=\frac{c^2\psi(\eta)}{h(B_c(\eta))}.
\end{equation} 
Hence, we assume that
\[
\ell_1:=\liminf_{\eta\to0} \psi(\eta)<\limsup_{\eta \to 0 } \psi(\eta)=:\ell_2.
\]
Let $\ell\in(\ell_1,\ell_2)$. Consider two sequences $\{\xi_n \}_n$ and $\{\sigma_n \}_n$ such that $\xi_n\to 0$ and $\sigma_n\to 0$ with $\lim_{n\to\infty} \psi(\xi_n)=\ell=\lim_{n\to\infty} \psi(\sigma_n)$, $\dot \psi(\xi_n)>0$, and $ \dot \psi(\sigma_n)<0$, for all $n\in\mathbb{N}$. Thanks to the following computation
\[
\dot \psi (\eta) =\frac{-1 - \dot B_c(\eta) - \psi(\eta) [\dot g(\eta)f(\eta,B_c(\eta)) + g(\eta) (f'_{\eta}(\eta,B_c(\eta)) + f'_{\beta}(\eta,B_c(\eta))\dot B_c(\eta))]} {g(\eta)f(\eta,B_c(\eta))},
\]
and $(A_1)$--$(A_2)$, we obtain
\[\begin{split}
-1 - \dot B_c(\xi_n)- \psi(\xi_n) \Big(\dot g(\xi_n)f(\xi_n,B_c(\xi_n)) &+ g(\xi_n) (f'_{\eta}(\xi_n,B_c(\xi_n)) \\
&+ f'_{\beta}(\xi_n,B_c(\xi_n))\dot B_c(\xi_n))\Big) >0, \quad \forall n\in\mathbb{N},\\
-1 - \dot B_c(\sigma_n) - \psi(\sigma_n) \Big(\dot g(\sigma_n)f(\sigma_n,B_c(\sigma_n)) &+ g(\sigma_n) (f'_{\eta}(\sigma_n,B_c(\sigma_n)) \\
&+ f'_{\beta}(\sigma_n,B_c(\sigma_n))\dot B_c(\sigma_n))\Big)<0, \quad \forall n\in\mathbb{N}.
\end{split}\]
From $(A_1) $ we get that $ f'_{\eta} (0,1) \geq L_1 > 0 $ and $ f'_{\beta}(0,1)=0 $ thus, by passing to the limit as $n\to+\infty$ and using~\eqref{eq-psi1}, we deduce that
\[\lim_{n \to +\infty} = -1 -\frac{c^2}{h(1)}\ell\geq0,\quad
\lim_{n \to +\infty} = -1-\frac{c^2}{h(1)}\ell\leq0.
\]
Thus, it follows $\ell=-\frac {h(1)}{c^2}$, which contradicts the arbitrariness of $\ell$. Hence, $\lim_{\eta\to 0} \psi(\eta) $ exists. According to $(A_1)$--$(A_3) $ and Propositions \ref{th-beta} and \ref{prop-3} we have that $ \dot B_c(\eta) \le 0 $ for every $ \eta \in (0,1). $ Moreover, $ B_c(\eta) - 1 > - \eta $ and $ B_c(0)=1.$ Hence, by the L'Hôpital rule, $\dot B_c(0)=-1$. This ends the proof.
\end{proof}

\begin{remark}
In Section \ref{sub-uniq}, the unique semi-wavefront $ (\eta,\beta) $ of~\eqref{eq-sys}--\eqref{eq-bc} is associated with the desingularized first-order problem~\eqref{eq-manifold}, that we linearized around the equilibrium point $ (0,1). $ The property $ \dot B_c(0)=-1 $ is consistent with the proof of Theorem~\ref{th-uniq} where we proved that the solution of the linearized problem leaves the equilibrium point along the central manifold which is tangent to the eigenspace generated by the vector $ (1,-1). $ 
\end{remark}

We now provide a necessary and sufficient condition so that a speed is admissible by exploiting the behavior of the solution of~\eqref{eq-bb}. 
\begin{lemma}\label{lem-cc}
$ c \in\mathcal{C}$ if and only if $B_c(1)=0.$	
\end{lemma}

\begin{proof}
If $ c \in\mathcal{C}$, then, by Corollary~\ref{cor-cond}, $\sigma=\tau$. Thus, $B_c(1)=\beta(\xi(1))=\beta(\sigma)=0$. On the other hand, if $ c \not\in\mathcal{C}$, then $\sigma<\tau$ and $\beta(\sigma)>0$. Therefore, $B_c(1)=\beta(\xi(1))=\beta(\sigma)>0$, concluding the proof.
\end{proof}

The following result describes the monotonicity of the solution $B_c(\eta)$ to problem~\eqref{eq-bb} with respect to the wave speed.
\begin{lemma}\label{lemma-c1}
For all $c_1,c_2 > 0 $, if $c_1<c_2$, then $B_{c_1}(\eta)>B_{c_2}(\eta)$ for every $\eta\in(0,1)$.
\end{lemma}

\begin{proof} We first of all prove that, for every $c>0$, given $B_c$ the associated solution of~\eqref{eq-bb}, then, for every $ \eta $ in a right neighbourhood of $ 0 $, 
\begin{equation}\label{eq-x2}
B_c(\eta)=1-\eta+\frac{1}{c^2}h(1)\dot{g}(0)f'_\eta(0,1)\eta^2+o(\eta^2).
\end{equation}
Consider 
\begin{equation}\label{eq-x1}
\begin{split}
\lim_{\eta\to0}\frac{B_c(\eta)-1+\eta}{\eta^2}&
=\lim_{\eta\to0}\left(-\frac{c^2(1-\eta-B_c(\eta))}{g(\eta)h(B_c(\eta))f(\eta,B_c(\eta))}\cdot
\frac{g(\eta)h(B_c(\eta))f(\eta,B_c(\eta))}{c^2\eta^2} \right)\\
&=-\frac{1}{c^2}\lim_{\eta\to0}\dot B_c(\eta)h(B_c(\eta))\frac{g(\eta)}{\eta}
\cdot\frac{f(\eta,B_c(\eta))}{\eta}. 
\end{split}\end{equation}
Using Taylor expansion, we compute 
\[\frac{f(\eta,B_c(\eta))}{\eta}=f'_\eta(0,1)+\frac{o(\sqrt{\eta^2+(B_c(\eta)-1)^2})}{\eta},\]
since $f(0,1)=0$ and $f'_\beta(0,1)=0$. From Lemma \ref{lem-dotB}, $(A_2) $ and~\eqref{eq-x1} and since $f'_\eta(0,1)\geq L_1>0$, it follows
\begin{equation} \label{ddB0}
\lim_{\eta\to0}\frac{B_c(\eta)-1+\eta}{\eta^2}=\frac{1}{c^2}h(1)\dot{g}(0)f'_\eta(0,1)>0,
\end{equation}
proving~\eqref{eq-x2}. If $0<c_1<c_2$, then ${1}/{c_1^2}>{1}/{c_2^2}$, and so we deduce that $B_{c_1}(\eta)>B_{c_2}(\eta)$ for every $\eta$ in a right neighbourhood of $0$. By contradiction, let us assume that there exists $\bar\eta\in(0,1)$ such that $B_{c_1}(\bar\eta)-B_{c_2}(\bar\eta)=0$ and $B_{c_1}(\eta)-B_{c_2}(\eta)>0$ for all $\eta\in(0,\bar\eta)$. Thus, by Propositions \ref{th-beta}, it follows
\[
\dot B_{c_1}(\bar\eta)- \dot B_{c_2}(\bar\eta)=\frac{1-\bar\eta- B_{c_1}(\bar\eta)}{g(\bar\eta)h(B_{c_1}(\bar\eta))f(\bar\eta,B_{c_1}(\bar\eta))}(c_1^2-c_2^2)>0,
\]
which is a contradiction.
\end{proof}

The following gives the connection of the set $\mathcal{C}$.
\begin{proposition} \label{prop-conn}
For all $c_1\in\mathcal{C}$ and $c>c_1$, then $c\in\mathcal{C}$.
\end{proposition}

\begin{proof}
Let us fix $c_1\in\mathcal{C}$ and take $c>c_1$. Then we consider the associated solutions $B_{c_1}$ and $B_c$ of~\eqref{eq-bb}. 
By Lemma~\ref{lem-cc}, $B_{c_1}(1)=0$. Moreover, since $c>c_1>0$, by using Lemma~\ref{lemma-c1}, we have $B_{c}(\eta)<B_{c_1}(\eta)$ for every $\eta\in(0,1)$. From this condition, Corollary~\ref{cor-cond} and Remark~\ref{rem-figures}, 
we get  $B_{c}(1)= 0$. Accordingly to Lemma~\ref{lem-cc}, we obtain $c\in\mathcal{C}$.
\end{proof}

\begin{lemma}\label{lem-B1}
For every $c\in\mathcal{C}$, $\dot B_c(1)=-\infty$ if and only if $\displaystyle\lim_{\xi\to\tau^-}\beta'(\xi)=-\frac{c}{g(1)\dot h(0)}$ and $\dot B_c(1)=-1$ if and only if $\displaystyle\lim_{\xi\to\tau^-}\beta'(\xi)=0$.
\end{lemma}
\begin{proof}
By the definition of $B_c(\eta)$, 
\[
\lim_{\eta\to 1} \dot B_c(\eta)=\lim_{\xi\to\tau^-}\frac{\beta'(\xi)}{\eta'(\xi)}=\lim_{\xi\to\tau^-}\frac{1}{\frac{\eta'(\xi)}{\beta'(\xi)}}.
\]
The thesis is then a straightforward consequence of Lemma \ref{lemma1} and Remark \ref{rem-dotz}.
\end{proof}

According to Corollary \ref{cor-alt}  and Lemma \ref{lem-B1}, the classification of the wavefront can be based on the value of $\dot B_c(1)$ as follows.  
\begin{corollary} \label{cor-cw} 
For every $c\in\mathcal{C}$, the wavefront is sharp if and only if $\dot B_c(1)=-\infty$ and it is classic if and only if $\dot B_c(1)=-1. $ 
\end{corollary}

\begin{lemma} \label{lem-us}
Let us suppose that there exists $ \tilde c \in \mathcal C $ such that the wavefront $ (\eta_{\tilde c},\beta_{\tilde c}) $ is sharp. Then, for every $ c \not= \tilde c, $ the wavefront $ (\eta_c, \beta_c) $ is classical. Moreover,  for every $ c > \tilde c, $ $ \tau_c = +\infty. $ 
\end{lemma}
\begin{proof}
We divide the proof into two steps.

\noindent
\textit{Step~$1$. Uniqueness of the sharp wavefront.} Let us suppose that there exists $ c \in \mathcal C $ such that $ (\eta_{c},\beta_{c}) $ is sharp. Then $ \dot B_{c}(1)=-\infty, $ hence, by De L'Hôpital rule,
\[ \lim_{\eta \to 1} \frac {1 - \eta} {B_{c}(\eta)} = \lim_{\eta \to 1} \frac {-1} {\dot B_{c}(\eta)} = 0. \]  
Moreover, since $f$ is differentiable in $(1,0)$, 
\[
f(1,0)=0=f'_\eta(1,0)\quad\hbox{and}\quad f'_\beta(1,0) \ge L_1 >0,
\]
thus, according to Lemma \ref{lem-cc}, for every $\eta$ in a left neighbourhood of $1$, we can write
\begin{equation}\label{f10}
f(\eta,B_{c}(\eta))=f'_\beta(1,0)B_{c}(\eta)+o(\eta-1+B_{c}(\eta)).
\end{equation}
Now, by De L'Hôpital rule and \eqref{f10}, it follows that
\begin{align*} 
\lim_{\eta\to 1}\frac{B^2_c(\eta)}{1-\eta} &=\lim_{\eta\to 1}\frac{2B_c(\eta)\dot{B}_c(\eta)}{-1}= -\frac{2c^2}{g(1)}\lim_{\eta\to 1}  \left[ \left(\frac{1-\eta}{B_c(\eta)}-1\right) \frac {B_c(\eta)}{h(B_c(\eta))} \frac {B_c(\eta)} {f(\eta,B_c(\eta))} \right] =\\
&=\frac{2c^2}{g(1)\dot{h}(0)}\lim_{\eta\to 1}\frac{1}{f'_\beta(1,0)+\frac{o(\eta-1+B_c(\eta))}{\eta-1+B_c(\eta)}\left( \frac{\eta-1}{B_c(\eta)} + 1 \right)} = \frac{2c^2}{g(1)\dot{h}(0)f'_\beta(1,0)}.
\end{align*}	
Thus, for all $\eta$ in a left neighbourhood of $1$,
\begin{equation}\label{B1}
B_c(\eta)=c\sqrt{\frac{2(1-\eta)}{g(1)\dot{h}(0)f'_\beta(1,0)}}+o(\sqrt{1-\eta}).
\end{equation}
Assume now the existence of two admissible speeds $c_1,c_2\in\mathcal{C}$ such that $\dot{B}_{c_1}(1)=\dot{B}_{c_2}(1)=-\infty$. Then, if $c_1<c_2$, by \eqref{B1},
\[
B_{c_1}(\eta)-B_{c_2}(\eta)<0
\]
for every $\eta$ in a left neighbourhood of~$1$. This, however, contradicts Lemma~\ref{lemma-c1}. Therefore, there exists at most a unique ${\tilde c}\in\mathcal{C}$ such that $ (\eta_{\tilde c},\beta_{\tilde c}) $ is sharp. 

\noindent
\textit{Step~$2$. For every $c>\tilde c$, $\tau_c=+\infty$.} 
Assume $c>\tilde c$. In this case, by the previous step, the wavefront $(\eta_c,\beta_c)$ associated with the speed $c$ is classical. Assume by contradiction, that $\tau_c<+\infty$. Then, by taking a suitable translation of $(\eta_c,\beta_c)$ and $(\eta_{\tilde c},\beta_{\tilde c})$, we can assume that $\tau:=\tau_c=\tau_{\tilde c}<+\infty$. First of all, by Corollary \ref{cor-alt} and by the L'Hôpital rule, 
\[
+\infty=\lim_{\xi\to \tau^-}\frac{\beta_{\tilde c}'(\xi)}{\beta_{c}'(\xi)}=\lim_{\xi\to \tau^-}\frac{\beta_{\tilde c}(\xi)}{\beta_{c}(\xi)}.
\]
From Lemma~\ref{lem-B1} and $(A_1)$, we have
\[\begin{split}
+\infty&=\lim_{\eta \to 1}\frac{\dot B_{\tilde c}(\eta)}{\dot B_{c}(\eta)}=\lim_{\xi \to \tau^-}\frac{\beta'_{\tilde c}(\xi)}{\eta'_{\tilde c}(\xi)}\frac{\eta'_{c}(\xi)}{\beta'_{c}(\xi)}=\lim_{\xi \to \tau^-}\frac{\beta_{\tilde c}(\xi)}{\beta_{c}(\xi)}\frac{\eta'_{c}(\xi)}{\eta'_{\tilde c}(\xi)}\\
&=\lim_{\xi \to \tau^-}\frac{\beta_{\tilde c}(\xi)}{\beta_{c}(\xi)}\frac{\tilde cf(\eta_{c}(\xi),\beta_{c}(\xi))}{cf(\eta_{\tilde c}(\xi),\beta_{\tilde c}(\xi))} = \frac {\tilde c} {c}  \lim_{\xi \to \tau^-}\frac{\beta_{\tilde c}(\xi)}{f(\eta_{\tilde c}(\xi),\beta_{\tilde c}(\xi))}\frac{f(\eta_{c}(\xi),\beta_{c}(\xi))} {\beta_{c}(\xi)}\leq\frac{\tilde c L_2}{c L_1},
\end{split}
\]
which is a contradiction, showing that $\tau_c=+\infty$. This ends the proof. 
\end{proof}

Now we aim to prove that the set of admissible speeds is closed. First of all, we define
\begin{equation}\label{eq-ccc}
c_0:=\inf \mathcal{C},
\end{equation}
and observe that, thanks to Remark~\ref{rem41}, $c_0$ is a positive real number and $c_{\sharp}\leq c_0\leq c_*.$

The following result says that the set $\mathcal{C}=[c_0,+\infty)$.
\begin{proposition} \label{prop-c0}
$c_0\in\mathcal{C}$ and $c_\sharp\leq c_0\leq c_*$. 
\end{proposition}

\begin{proof}
Consider the solution $B_{c_0}$ to equation~\eqref{eq-bb} when $c=c_0$. We will prove that $B_{c_0}(1)=0.$ Consequently, using Lemma~\ref{lem-cc},   $c_0\in\mathcal{C}$.

Let us take a decreasing sequence $\{c_n\}_n$  converging to $c_0$ and consider, for every $n$, the solution $B_{c_n}$ of~\eqref{eq-bb} corresponding to $ c=c_n. $ Then $ B_{c_n}(0)= 1 $ and $ \dot B_{c_n}(0)=-1, $ by Lemma \ref{lem-dotB}. According to~\eqref{eq-ccc} and Lemma~\ref{lem-cc}, for every $n$, $ c_n \in \mathcal C, $ and $ B_{c_n}(1)=0. $ Moreover, by Lemma \ref{lem-us} and Corollary \ref{cor-cw}, without loss of generality we can suppose that $ \dot B_{c_n}(1)=-1 $ for every $ n. $ Finally, taking a suitable translation, we can suppose that the wavefront $ (\eta_{c_n},\beta_{c_n}) $ satisfies $ \eta_{c_n}(0)=\frac 1 2 $ for every $ n. $

For every $\eta \in [0,1] $ and $ n, $ we have 
\[0 \leq B_{c_n}(\eta)\leq 1,
\]
and so $\{B_{c_n}\}_n$ is equibounded in $[0,1] $.

Let us now prove that $ \{ \dot B_{c_n} \}_n $ is equibounded too. Notice that $ \dot{B}_{c_n}(\eta) \leq 0 $ for every $ n$ and $\eta \in [0,1]. $ Suppose by contradiction that there exists a sequence $ \{ \eta_n \}_n $ such that $ \lim_{n \to +\infty} \dot B_{c_n}(\eta_n) = - \infty. $ Since $ \{\eta_n \}_n \subset [0,1] $, there exists a subsequence, still denoted as the sequence, and $ \overline {\eta} \in [0,1], $ such that $ \lim_{n \to +\infty} \eta_n = \overline {\eta}. $ Assume firstly that $ \overline {\eta} \in (0,1). $ Then, without loss of generality, we can find $ \epsilon > 0 $ such that $ \eta_n \in [\overline {\eta} - \epsilon, \overline {\eta} + \epsilon] \subset(0,1)$ for every $ n. $ Then, since $\{c_n\}_n$  is a decreasing sequence and $ B_{c_n} $ is a decreasing function, by Lemma \ref{lemma-c1} we deduce 
\[\begin{split}
\dot{B}_{c_n}(\eta_n)=& \frac{-c_n^2(B_{c_n}(\eta_n) + \eta_n - 1)}{g(\eta_n)h(B_{c_n}(\eta_n))f(\eta_n,B_{c_n}(\eta_n))}\\
\geq & \frac{-c_1^2}{L_1 (\overline {\eta} - \epsilon) B_{c_1}(\overline {\eta} + \epsilon) (\min_{[\overline{\eta} - \epsilon,\overline{\eta} + \epsilon ]}g) (\min_{[B_{c_1}(\overline{\eta} + \epsilon),1]}h)}, 
\end{split}\]
which contradicts $ \lim_{n \to +\infty} \dot B_{c_n}(\eta_n) = - \infty. $ Assume now that $ \overline {\eta} =0. $ Then, by the mean value theorem and \eqref{ddB0}, recalling that $\dot B_{c_n}(0) = -1, $ for every $ n $ there exists $ x_n \in (0,\eta_n) $ such that 
\[ \dot B_{c_n} (\eta_n)= \ddot B_{c_n}(x_n) \eta_n + \dot B_{c_n}(0) = \left [\frac{2}{c_n^2}h(1)\dot{g}(0)f'_\eta(0,1) + o_n(1)\right]  \eta_n -1, \] 
which again contradicts $ \lim_{n \to +\infty} \dot B_{c_n}(\eta_n) = - \infty. $ Finally, assuming that  $ \overline {\eta} =1 $ with a similar reasoning we get, for every $ n , $  
\[ \dot B_{c_n} (\eta_n)= \left [-\frac{2}{c_n^2}g(1)\dot{h}(0)f'_\beta(1,0) + o_n(1)\right] (\eta_n -1) -1, \] 
hence again a contradiction. So $\{\dot B_{c_n}\}_n$ is equibounded in $[0,1]$. Hence, $\{B_{c_n}\}_n$ is equicontinuous in $[0,1]$ and by the Ascoli-Arzel\`a theorem there is a subsequence, that we still denote by $\{B_{c_n}\}_n$, converging uniformly on $[0,1]$ to a function $Y\in C[0,1]$. Since we have, for every $ \eta, \hat{\eta} \in (0,1), $ with $ \hat{\eta}<\eta$,
\[
B_{c_n}(\eta)=B_{c_n}(\hat{\eta})+\int_{\hat{\eta}}^\eta \dot B_{c_n}(s)\,\mathrm{d}s =B_{c_n}(\hat{\eta})+\int_{\hat{\eta}}^\eta \frac{-c_n^2(1-s-B_{c_n}(s))}{g(s)h(B_{c_n}(s))f(s,B_{c_n}(s))} \,\mathrm{d}s,
\]
by passing to the limit as $n\to+\infty,$
\[
Y(\eta)=Y(\hat{\eta})+\int_{\hat{\eta}}^\eta \frac{-c_0^2(1-s-Y(s))}{g(s)h(Y(s))f(s,Y(s))} \,\mathrm{d}s.
\] 
Moreover, by Proposition \ref{prop-conv}, 
\[ \eta_{c_0} (0) = \lim_{n \to +\infty}   \eta_{c_n} (0) = \frac 1 2 \]
and
\[ \beta_{c_0} (0) = \lim_{n \to +\infty}   \beta_{c_n} (0). \]
Therefore
\begin{align*}
B_{c_0}(1/2) = & \beta_{c_0}(\xi_{c_0}(1/2)) = \beta_{c_0}(0) \\
= & \lim_{n \to + \infty} \beta_{c_n}(0)= \lim_{n \to +\infty}\beta_{c_n}(\xi_{c_n}(1/2)) = \lim_{n \to + \infty} B_{c_n}(1/2) = Y(1/ 2), 
\end{align*}
thus we obtain that $Y$ is the solution of 
\[\begin{cases}
\dot{Y}(\eta)=\dfrac{-c_0^2(1-\eta-Y(\eta))}{g(\eta)h(Y(\eta))f(\eta,Y(\eta))},\\[10pt]
Y({1}/{2})=B_{c_0}({1}/{2}).
\end{cases}\]
By the uniqueness of the solution of the Cauchy problem, it follows that $Y(\eta)=B_{c_0}(\eta)$, for all $\eta\in[0,1)$. Moreover 
\[ 
\lim_{\eta \to 1^-} B_{c_0} (\eta) = \lim_{ \eta \to 1^-} Y(\eta) = \lim_{\eta \to 1^-} \lim_{n \to + \infty} B_{c_n}(\eta) =  \lim_{n \to + \infty} B_{c_n}(1)=0. \]
For the estimates, we exploit Proposition~\ref{prop28}, Theorem~\ref{th-positive}, and Remarks~\ref{rem-star} and~\ref{rem41} to conclude the proof.
\end{proof}

\begin{proposition}\label{prop-B1inf}
$\dot B_{c_0}(1)=-\infty.$
\end{proposition}

\begin{proof}
Let us take an increasing sequence $\{c_n\}_n$ converging to $c_0$, and consider, for every $n$, the solution $B_{c_n}$ of~\eqref{eq-bb} corresponding to $ c=c_n. $ Then $ B_{c_n}(0)= 1 $ and $ \dot B_{c_n}(0)=-1, $ by Lemma \ref{lem-dotB}. According to Proposition \ref{prop-conn} and Lemma \ref{lem-cc}, for every $n$, $ c_n  \not\in\mathcal{C}$, i.e. $ B_{c_n}(1)>0.$

Since $\{c_n\}_n$ is increasing, by Lemma~\ref{lemma-c1}, we deduce that $\{B_{c_n}\}_n$ is a decreasing sequence of functions in $C[0,1]$. Let $Z$ be its limit as $n\to+\infty.$ Consider $\hat{\eta},\eta\in(0,1)$ with $\hat{\eta}<\eta$. For all $s\in[\hat{\eta},\eta]$, we have
\[0\geq \dot{B}_{c_n}(s)= \frac{-c_n^2(B_{c_n}(s) + s - 1)}{g(s)h(B_{c_n}(s))f(s,B_{c_n}(s))}\geq\frac{-c_0^2}{L_1 \hat{\eta} B_{c_0}(\eta) (\min_{[\hat{\eta},\eta ]}g) (\min_{[B_{c_0}(\eta),1]}h)}.
\]
Thus, by arguing as in proof of Proposition~\ref{prop-c0}, we have that $Z$ is the solution of 
\[\begin{cases}
\dot{Z}(\eta)=\dfrac{-c_0^2(1-\eta-Z(\eta))}{g(\eta)h(Z(\eta))f(\eta,Z(\eta))},\\[10pt]
Z({1}/{2})=B_{c_0}({1}/{2}).
\end{cases}\]
Since $B_{c_n}(0)=1$ for every $n$, by the uniqueness of the solution of the Cauchy problem, it follows that $Z(\eta)=B_{c_0}(\eta)$ for every $\eta\in[0,1)$. Next, we claim that $Z(1)=0$. Therefore, let us assume by contradiction that $\lim_{n\to+\infty}B_{c_n}(1)=\theta>0.$ Since $\{B_{c_n}\}_n$ is a decreasing sequence of decreasing functions, for every $n$ and $\eta\in[0,1)$, $B_{c_n}(\eta)>B_{c_n}(1)>\theta.$ Thus, by passing to the limit as $n\to+\infty$, we obtain $B_{c_0}(\eta)\geq\theta$ for every $\eta\in[0,1)$, which is a contradiction because $B_{c_0}(1)=0$. We thus conclude that $\lim_{n\to+\infty}B_{c_n}(1)=0$, so that $Z(1)=0$ implying $Z(\eta)=B_{c_0}(\eta)$ for every $\eta\in[0,1].$ It follows that the decreasing sequence $\{B_{c_n}\}_n$ converges to the continuous function $B_{c_0}$ and so, by the Dini's theorem, the convergence is uniform in $C[0,1]$. In particular, $\{B_{c_n}\}_n$ is equicontinuous in $[0,1]$, i.e. denoting its modulus of equicontinuity by $\omega\colon[0,+\infty)\to[0,+\infty)$, for every $\delta\geq0$ and for every $n$, if $\eta_1,\eta_2\in[0,1]$ and $|\eta_1-\eta_2|\leq\delta$, then 
\begin{equation}\label{eq-equicon}
\left| B_{c_n}(\eta_1)-B_{c_n}(\eta_2)	\right|\leq \omega(\delta).
\end{equation}
Thus, by passing to the limit as $n\to+\infty$, we also have
\[
\left| B_{c_0}(\eta_1)-B_{c_0}(\eta_2)	\right|\leq \omega(\delta).
\]
Let us suppose by contradiction that $\dot B_{c_0}(1)=-1$. It follows that $B_{c_0}$ is of class $C^1$ in $[0,1]$, and so there exists $D>0$ such that $|\dot{B}_{c_0}(\eta)|\leq D$ for all $\eta\in[0,1]$ yielding $\omega(\delta)\leq D\delta$ for every $\delta<1.$ From~\eqref{eq-equicon}, for every $\delta<1,$ it thus follows
\[
\frac{\left| B_{c_n}(1-\delta)-B_{c_n}(1)	\right|}{\delta}\leq D.
\]
By passing to the limit as $\delta\to0^+$, we obtain $|\dot{B}_{c_n}(1)|\leq D$ for all $n$. On the other hand, since $\lim_{n\to+\infty}B_{c_n}(1)=0$, we have 
\[
\lim_{n\to+\infty}\dot{B}_{c_n}(1)=\lim_{n\to+\infty}\frac{-c_n^2 B_{c_n}(1)}{g(1)h(B_{c_n}(1))f(1,B_{c_n}(1))}=\lim_{n\to+\infty}\frac{-c_n^2}{g(1)\frac{h(B_{c_n}(1))}{B_{c_n}(1)}f(1,B_{c_n}(1))}=-\infty,
\]
which is a contradiction, and the thesis follows.
\end{proof}

\begin{corollary}\label{cor-clas}
The following hold:
\begin{enumerate}[nosep,wide=0pt,  labelwidth=20pt, align=left]
\item[$(i)$] if $c=c_0$, then $\tau<+\infty$ and $(\eta,\beta)$ is sharp;
\item[$(ii)$] if $c>c_0$, then $\tau=+\infty$ and $(\eta,\beta)$ is classical.
\end{enumerate}
\end{corollary}

\begin{proof}
To prove the thesis, we exploit Propositions \ref{prop-c0}, ~\ref{prop-B1inf} and \ref{prop-conn}, Corollary \ref{cor-cw}, and Lemmas~\ref{lem-B1} and \ref{lem-us}. 
\end{proof}

\noindent

%%================================
{
\bibliographystyle{elsart-num-sort}
\bibliography{biblio_MHST}
%\nocite*{}
\bigskip\Addresses
}

%%-_-_-_-_-_-_-_-_-_-_-_-_-_-_-_-_
\end{document}